\documentclass[11pt,onecolumn,longbibliography,
 amsmath,amssymb,superscriptaddress,tightenlines]{revtex4-2}
\usepackage{graphicx}
\usepackage[colorlinks=true,citecolor=blue, linkcolor=blue, urlcolor=blue]{hyperref}
\usepackage{epsfig}
\usepackage{amsthm}
\usepackage{xcolor}
\usepackage{url}
\numberwithin{equation}{section}
\usepackage{float}
\usepackage{graphicx}
\usepackage{subcaption}
\usepackage[labelfont=bf]{caption}
\renewcommand{\Re}{\mathop{\rm Re}}
\renewcommand{\Im}{\mathop{\rm Im}}
\newcommand{\R}{\mathfrak{t}}
\newcommand{\s}{\mathfrak{s}}
\newcommand{\p}{\mathfrak{p}}
\newcommand{\C}{\mathbb{C}}
\newcommand{\N}{\mathbb{N}}

\newcommand{\M}{\mathcal{M}}
\renewcommand{\Re}{\mathop{\rm Re}}
\renewcommand{\Im}{\mathop{\rm Im}}
\newtheorem{theorem}{Theorem}[section]
\newtheorem{lemma}[theorem]{Lemma}

\newtheorem{corollary}[theorem]{Corollary}

\newtheorem{remark}[theorem]{Remark}

\usepackage{tocloft}  
\usepackage{titlesec} 

\makeatletter
\renewcommand{\p@subsection}{}
\renewcommand{\p@subsubsection}{}
\makeatother

\begin{document}

\title{Differential equations for a class of  semiclassical orthogonal polynomials on the unit circle}

\author{Cleonice F. Bracciali} 
\email{cleonice.bracciali@unesp.br}
\affiliation{Departamento de Matemática, IBILCE, UNESP-Universidade Estadual Paulista, São José do Rio Preto, 15054-000, SP, Brazil.}

\author{Karina S. Rampazzi} 
\email{karina.rampazzi@unesp.br}
\affiliation{Faculdade SESI de Educação, Sorocaba, 18040-350, SP, Brazil.}

\author{Luana L. Silva Ribeiro}
\email{luana.ribeiro@unifei.edu.br (Corresponding author)}
\affiliation{Instituto de Ciências Puras e Aplicadas, UNIFEI-Universidade Federal de Itajubá, Itabira, 35903-087, MG, Brazil}

\begin{abstract}
We consider semiclassical orthogonal polynomials on the unit circle associated with a weight
function that satisfy a Pearson-type differential equation involving two polynomials of degree at most three. Structure relations and difference equations for these orthogonal polynomials are found, and, as a consequence, explicit first and second order differential equations are derived. Among the applications, differential equations for a family of polynomials  that generalizes the Jacobi polynomials on the unit circle and the modified Bessel polynomials are established. It is also shown that in some cases the Verblunsky coefficients satisfy a discrete Painlevé II equation.
\vspace{0.5cm}

\vspace{0.2cm}

\end{abstract}

\maketitle

\section{Introduction}
\label{sec1}
Let $\mu$ be a positive measure on the unit circle $\mathbb{T}=\{z\in\C:|z|=1\}$, parametrized by $z=e^{i\theta}$, and let $\{\Phi_n\}_{n\in\N}$ be the sequence of \textit{monic orthogonal polynomials on the unit circle} (MOPUC) satisfying the orthogonality condition $ \langle \Phi_n,\Phi_m \rangle =\kappa^{-2}_n\delta_{m,n}$, $m,n\in\mathbb{Z}_+=\{0,1,2,\ldots\}$, where
\begin{align}\label{General-Orthogonality}
    \langle f,g \rangle =\int_{0}^{2\pi}f(e^{i\theta})\overline{g(e^{i\theta})}\mbox{d}\mu(e^{i\theta})
\end{align}
is the standard inner product on $\mathbb{T}$, see \cite{Simon2005}. In this work we adopt the following notation for the MOPUC, 
\begin{equation}\label{monic-OPUC}
    \Phi_n(z)=z^n+\ell_{n,n-1} z^{n-1}+\ell_{n,n-2} z^{n-2}+\cdots +\ell_{n,1} z +\ell_{n,0}, \ n\in \mathbb{Z}_+,
\end{equation}
where $\alpha_{n-1}=-\overline{\ell}_{n,0}$ is the Verblunsky coefficient with  $\alpha_{-1}=-1$. Also $\kappa_n^2/\kappa_{n+1}^2=1-|\alpha_n|^2.$ These polynomials satisfy the Szeg\H{o}'s  recurrences 
\begin{equation}\label{Szego-recurrence}
\begin{aligned}
    &\Phi_{n}(z) = z\Phi_{n-1}(z)-\overline{\alpha}_{n-1} \Phi_{n-1}^{\ast}(z), \\ 
    &\Phi_{n}(z) = (1 - |\alpha_{n-1}|^2) z\Phi_{n-1}(z)  - \overline{\alpha}_{n-1} \Phi_{n}^{\ast}(z),
\end{aligned}
\end{equation}
where $\Phi_n^{\ast}(z)=z^n\overline{\Phi_{n}(1/\overline{z})}$ is the reversed polynomial. The above recurrences imply 
\begin{equation}
\begin{aligned} \label{prop-coef-n-1}
    & \ell_{n,1}=-(\overline{\alpha}_{n-1}\overline{\ell}_{n-1,n-2}+\overline{\alpha}_{n-2})=-[(1-|\alpha_{n-1}|^2)\overline{\alpha}_{n-2}+\overline{\alpha}_{n-1}\overline{\ell}_{n,n-1}], \ n\geqslant 1,  \\
    & \ell_{n,n-1}=\ell_{n-1,n-2}+\overline{\alpha}_{n-1}\alpha_{n-2}, \ n\geqslant 1 \ \mbox{with} \ \ell_{0,-1}=0,  \\
    & \ell_{n,n-2}=\ell_{n-1,n-3}-\overline{\alpha}_{n-1}\overline{\ell}_{n-1,1},  \ n\geqslant 2 \ \mbox{and} \ \ell_{1,-2}=0.
\end{aligned}
\end{equation}

One goal of this work is to establish structure relations (or differential-recurrence relations) for orthogonal polynomials on the unit circle. These are finite relations involving the polynomials and their derivatives. Such relations appear in the theory of (classical and semiclassical) orthogonal polynomials on the real line  and also on the unit circle \cite{bonan1987orthogonal,branquinho2012structure,ISMAIL-JAT_2001,MARCELLAN2007537,maroni1991une,vanassche2018orthogonal}. We also cite applications in the framework of Sobolev inner products with
coherent pairs of measures \cite{DEJESUS2008482,iserles1991sobolev,marcellan2017sobolev}.

In general the classification of semiclassical orthogonal polynomials on the unit circle is based on semiclassical functionals \cite{AlfaroMarcellan1991,CachafeiroSuarez1997,Magnus2013}. Specifically, one considers the class of measures satisfying  $d\mu(e^{i\theta})=w(\theta)d\theta$, where $w$ is a positive \textit{semiclassical weight function}, i.e., $w$ satisfies the Pearson-type differential equation
\begin{equation}\label{Eq-Tipo-Pearson-1}
\frac{d}{d\theta} \left[ A (e^{i\theta})  w(\theta) \right]  = B(e^{i\theta})  w(\theta), \  0 \leqslant \theta \leqslant 2\pi.
\end{equation}
Here $A(z)$ and $B(z)$ are Laurent polynomials and $A(e^{i\theta})=0$ at singular points of $1/w$ (see \cite{AlfaroMarcellan1991,Magnus2013}).  
Furthermore, if $A(z)$ and $B(z) $ are polynomials and $w$ satisfy \eqref{Eq-Tipo-Pearson-1}, it is said that $w$ belongs to the class $(p,q)$ if $\deg A(z)=p$ and 
$\max\{p-1,\deg((p-1)A(z)+iB(z))\}=q$ \cite{CachafeiroSuarez1997,Suarez2001}. Note that the polynomials $A(z)$ and $B(z)$ are not unique. In fact, if the weight function belongs to the class $(p,q)$, it also belongs to the class $(p+1,q+1)$.

In this work the structure relations and difference equations for semiclassical MOPUC studied in \cite{Bracciali-Rampazzi-SRibeiro-JAT2023} are extended to the case where $A(z)$ and $B(z)$ are complex polynomials of degree at most three:
\begin{align}\label{defin-A-B-grau3}
    A(z)=a_{3}z^3+a_{2}z^2+a_{1}z+a_{0} \ \ \mbox{and} \ \ B(z)=b_{3}z^3+b_{2}z^2+b_{1}z+b_{0}.
\end{align}
As consequence, we obtain explicit first and second order differential equations for the related semiclassical polynomials. We note that differential equations for the MOPUC are an important tool for their description, and a method to find such equations can be found, for instance, in \cite{ISMAIL-JAT_2001}. We also cite \cite{Branquinho2009}, where the authors identify differential equations for the MOPUC using the Carathéodory function of the corresponding measure.
Our findings for the first and second order differential equations are obtained using a different approach from the ones described in \cite{ISMAIL-JAT_2001} and \cite{Branquinho2009}, and the equations are explicitly given in terms of the polynomials $A(z)$, $B(z)$, and the coefficients $\ell_{n,k}$ defined in \eqref{monic-OPUC}.

Furthermore, our results are used to find difference equations for the Verblunsky coefficients, including the discrete Painlevé II equation
\begin{align}\label{tipo-painleve-absolute-value}
\alpha_n+\alpha_{n-2}=-\frac{2}{t}\frac{[\lambda +(\lambda+n)\alpha_{n-1}]}{1-\alpha_{n-1}^2}, \ \ n \geqslant 2,
\end{align} 
 where $t\neq 0$, $\lambda> -1/2$. Equation \eqref{tipo-painleve-absolute-value} arises when one considers MOPUC associated with the weight function $w(\theta)=e^{t\cos(\theta)}[\sin^2(\theta/2)]^{\lambda}$. If $\lambda=0$, this finite difference relation appeared in \cite{periwal1990unitary} in connection to modified Bessel polynomials \cite{ISMAIL-JAT_2001,vanassche2018orthogonal}.
 Further results involving orthogonal polynomials and Painlevé equations (continuous and discrete) can be found for instance in \cite{Filipuk2018,MAGNUS-1995,Magnus1999,vanassche2018orthogonal} and references therein. Discrete Painlevé equations are also found in quantum gravity and random partitions models \cite{fokas1991discrete,ChouteauTarricone2023}.

Among the applications, novel differential equations for families of semiclassical orthogonal polynomials are found, including
the MOPUC with respect to the weight function
\begin{equation*}
w(\theta)= e^{-\eta \theta}[\sin^2(\theta/2)]^{\lambda}[\cos^2(\theta/2)]^{\beta}, \  0 \leqslant \theta \leqslant 2\pi,
\end{equation*}
where $\eta\in\mathbb{R}$, $\lambda>-1/2$ and $\beta>-1/2$. This measure gives rise to several known examples of MOPUC. When $\beta=0$, the corresponding polynomials were considered in \cite{Ranga-PAMS2010}, and if $\eta=0$ the Jacobi polynomials on the unit circle \cite{ISMAIL-JAT_2001,badkov1987systems} are found. Also, when $\eta=\beta=0$ we obtain the well-known weight function of the circular Jacobi polynomials \cite{ISMAIL-JAT_2001}. The case where $\beta\lambda\eta\neq 0$ was considered in \cite{Bracciali-Rampazzi-SRibeiro-JAT2023}.

This work is organized as follows. Section \ref{sec2} presents the structure relations for the MOPUC and difference equations for the Verblunsky coefficients. In Section \ref{Sec3} first and second order differential equations for $\Phi_n$ and $\Phi_n^*$ are presented. Section \ref{Sec4} contains applications of the results by presenting differential equations for several families of MOPUC. Finally, Section \ref{final-section} closes the analysis with final remarks and numerical simulations to generate the Verblunsky coefficients.

\section{Structure relations for the MOPUC in the class $(p,q)$ with $p,q \leqslant 3$}
\label{sec2}

Let $A(z)$ e $B(z)$ be (complex) polynomials of degree at most $d\geqslant 1$ and $w$ be such that \eqref{Eq-Tipo-Pearson-1} holds. A straightforward use of integration properties and \eqref{Eq-Tipo-Pearson-1} imply that 
\begin{align}\label{Eq-A-phi-1}
\langle A(z)\Phi_n',z^{k} \rangle & =\langle [iB(z)+(k+1)A(z)]\Phi_{n},z^{k+1} \rangle - i\Phi_n(1)A(1)[w(2\pi)-w(0)]. 
\end{align}
This type of property was explored in \cite{Magnus2013} when $w(2\pi)=w(0)$ and it was recently explored in \cite{Bracciali-Rampazzi-SRibeiro-JAT2023} for $d=2$. If $d=3$ we obtain
\begin{align}\label{Eq-A-phi-2}
\langle A(z)\Phi_n',z^{k} \rangle=- i\Phi_n(1)A(1)[w(2\pi)-w(0)], \ \ k=2,\ldots,n-2.
\end{align}
This is the basic relation to establish a structure relation for MOPUC in Theorem \ref{Thm1-main}.  

\begin{theorem}\label{Thm1-main} Let $w$ be a weight function for which \eqref{Eq-Tipo-Pearson-1} holds with $A(z)$ and $B(z)$ as in \eqref{defin-A-B-grau3}. If $A(1)[w(2\pi)-w(0)]=0$ then the associated monic orthogonal polynomials satisfy the structure relation for $n\geqslant 3$
\begin{equation*}
A(z)\Phi_{n}'(z)=na_{3}\Phi_{n+2}(z)+\s_{n,n+1}\Phi_{n+1}(z)+\s_{n,n}\Phi_{n}(z)+\s_{n,n-1}\Phi_{n-1}(z)+[\p_{n,n}z+\R_{n,n}]\Phi_{n}^{\ast}(z),  
\end{equation*}
where 
\begin{eqnarray*} 
\p_{n,n} &=& (ib_3+2a_3) \overline{\alpha}_n,   \quad 
 \label{pnn} \\[0.5ex] 
\R_{n,n} &=& (ib_3+a_3) \overline{\alpha}_{n+1} (1 - |\alpha_n|^2) + (ib_2+a_2 + a_3\ell_{n+1,n} ) \overline{\alpha}_{n}, 
\label{tnn}  \\[0.5ex] 
\s_{n,n-1} &=&    (ib_0+na_0)  (1 - |\alpha_{n-1}|^2), 
\label{snn-1} \\[0.5ex] 
\s_{n,n} &=&  ib_1+(n+1)a_1 - a_0\overline{\ell}_{n,n-1} - [ib_0+(n+1)a_0]  \alpha_n \overline{\alpha}_{n-1}, 
\label{snn}  \\[0.5ex]
\s_{n,n+1} &=& n a_2  + a_3[(n-1)\ell_{n,n-1}  -n \ell_{n+2,n+1}]  +(ib_3+2a_3) \overline{\alpha}_n\alpha_{n-1}. 
\label{snn+1-1form}   
\end{eqnarray*}
Here $\alpha_n$ is the Verblunsky coefficient and $\ell_{i,i-1}$, $i=n+2,n+1,n$, are as in \eqref{monic-OPUC}. 
\end{theorem}
\begin{proof}
The polynomial
$
A(z)\Phi_{n}'(z) - n a_3\Phi_{n+2}(z) -\p_{n,n} z\Phi_{n}^{\ast}(z) - \R_{n,n} \Phi_{n}^{\ast}(z)
$
has degree at most $n+1$ and thus it can be written as
\begin{align}\label{defAdevphi}
A(z)\Phi_{n}'(z) - n a_3 \Phi_{n+2}(z) -\p_{n,n} z\Phi_{n}^{\ast}(z) - \R_{n,n} \Phi_{n}^{\ast}(z) =
\sum_{j=0}^{n+1} \s_{n,j} \Phi_{j}(z). 
\end{align}
Then we have $
\langle A\Phi_{n}',1 \rangle
- n a_3 \langle  \Phi_{n+2},1 \rangle
- \p_{n,n} \langle  z\Phi_{n}^{\ast},1 \rangle
- \R_{n,n} \langle \Phi_{n}^{\ast},1 \rangle
= \sum_{j=0}^{n+1} \s_{n,j} \langle \Phi_{j}, 1 \rangle
$. Using the identities  $ \langle  z\Phi_{n}^{\ast},1 \rangle=-\ell_{n+1,n}\langle \Phi_n,\Phi_n\rangle$, $\langle \Phi_{n}^{\ast},1 \rangle=\langle \Phi_n,\Phi_n\rangle$, and the orthogonality of the MOPUC we find
$$
\langle A\Phi_{n}',1 \rangle
+ (\p_{n,n} \ell_{n+1,n} - \R_{n,n})\langle\Phi_n,\Phi_n\rangle = \s_{n,0} \langle 1,1 \rangle. 
$$
Let $\p_{n,n}$ and $\R_{n,n}$ be such that $(\R_{n,n}-\p_{n,n} \ell_{n+1,n})\langle\Phi_n,\Phi_n\rangle=\langle A\Phi_{n}',1 \rangle$, it follows 
that $\s_{n,0}=0$. Set $\p_{n,n}=\langle A\Phi_{n}',z \rangle/ \langle \Phi_{n}^{\ast}, 1 \rangle = (ib_3+2a_3)\overline{\alpha}_n$. By requiring that 
$(\R_{n,n}-\p_{n,n} \ell_{n+1,n})\langle\Phi_n,\Phi_n\rangle=\langle A\Phi_{n}',1 \rangle$ and by letting $k=1 $ in \eqref{Eq-A-phi-1} we deduce
$$
\langle A\Phi_{n}',1 \rangle 
= (ib_3+a_3) \langle z^2 \Phi_{n}, 1 \rangle + (ib_2+a_2) \langle z \Phi_{n}, 1 \rangle, 
$$
and since $\langle z^2 \Phi_{n}, 1 \rangle =[-\overline{\alpha}_n\ell_{n+1,n}+\overline{\alpha}_{n+1}(1-|\alpha_n|^2)]\langle\Phi_n,\Phi_n\rangle$
and $\langle z \Phi_{n}, 1 \rangle=\overline{\alpha}_n\langle\Phi_n,\Phi_n\rangle$, we conclude that 
\begin{align}
\R_{n,n}=\p_{n,n}\ell_{n+1,n}+(ib_3+a_3)\overline{\alpha}_{n+1}(1-|\alpha_n|^2)+[ib_2+a_2-\ell_{n+1,n}(ib_3+a_3)]\overline{\alpha}_n. 
\end{align}
Therefore, $\R_{n,n}$ is found from $\p_{n,n}$.

From \eqref{defAdevphi}, since $\s_{n,0}=0$, we obtain 
$
\langle A\Phi_{n}',z \rangle
- \p_{n,n} \langle  \Phi_{n}^{\ast} , 1 \rangle
= \s_{n,1} \langle \Phi_{1}, z \rangle.
$
Thus we obtain that $\s_{n,1}=0$ from the definition of $\p_{n,n}$. 

The inner product of \eqref{defAdevphi} with $z^k$ produces the identity
\begin{equation}\label{rel-estrutura-inner-product}
\langle A(z)\Phi_{n}',z^k \rangle - n a_3 \langle\Phi_{n+2},z^k \rangle -\p_{n,n} \langle z\Phi_{n}^{\ast},z^k \rangle - \R_{n,n} \langle\Phi_{n}^{\ast},z^k \rangle =
\sum_{j=2}^{n+1} \s_{n,j} \langle \Phi_{j},z^k \rangle. 
\end{equation}
Note that \eqref{Eq-A-phi-2} and the equation above, when successively applied for $k=2,\ldots,n-2$,  shows that $\s_{n,k}=0$ for $k=2\ldots,n-2$ in view of the orthogonality of the MOPUC. To conclude the proof we need only to find the $\s_{n,k}$ for $k=n-1,n,n+1$. 

When $k=n-1$ in \eqref{rel-estrutura-inner-product}, it follows that $\langle A\Phi_n',z^{n-1}\rangle=\s_{n,n-1}\langle \Phi_{n-1},z^{n-1}\rangle$ because $\s_{n,i}=0$ for $i=2,3,\ldots,n-2$. As $\langle A\Phi_n',z^{n-1}\rangle = (ib_{0}+na_{0})\langle \Phi_{n},\Phi_{n}\rangle$, the value of $\s_{n,n-1}$ follows. 

By taking $k=n$ in \eqref{rel-estrutura-inner-product} we get $\langle A\Phi_{n}',z^{n} \rangle
=
 \s_{n,n} \langle \Phi_{n}, z^n \rangle +
 \s_{n,n-1} \langle \Phi_{n-1}, z^{n} \rangle $.
To obtain $\s_{n,n}$ notice that $\langle \Phi_{n-1},z^n\rangle = -\overline{\ell}_{n,n-1} \langle \Phi_{n-1},\Phi_{n-1}\rangle$ and $\langle A\Phi_{n}',z^{n} \rangle = \{ib_1+(n+1)a_{1}-\overline{\ell}_{n+1,n}[ib_{0}+(n+1)a_0]\}\langle \Phi_n,\Phi_n \rangle$.

Finally, from \eqref{defAdevphi} we get $\langle A\Phi_{n}',\Phi_{n+1} \rangle = \s_{n,n+1}\langle \Phi_{n+1},\Phi_{n+1}\rangle+\p_{n,n}\langle z\Phi_n^{\ast},\Phi_{n+1}\rangle $, and  
$\s_{n,n+1}$ follows from 
$\langle A\Phi_{n}',\Phi_{n+1} \rangle=\{na_{2}+a_{3}[(n-1)\ell_{n,n-1}-n\ell_{n+2,n+1}]\}\langle \Phi_{n+1},\Phi_{n+1}\rangle$ and $\langle z\Phi_n^{\ast},\Phi_{n+1}\rangle=-{\alpha}_{n-1}\langle \Phi_{n+1},\Phi_{n+1}\rangle$. 
\end{proof}

The coefficients $\s_{n,n}$ and $\s_{n,n+1}$ appearing in the structure relation of Theorem \ref{Thm1-main} can also be expressed in an alternative forms, as shown in the next result.

\begin{corollary}\label{Coro1-formula} Under the hypotheses of Theorem \ref{Thm1-main}, the values of $\s_{n,n}$ and $\s_{n,n+1}$ can also be represented as 
\begin{align*} 
\s_{n,n}  = & \ 
 na_1 - a_2\ell_{n,n-1} + a_{3}[(\ell_{n+1,n})^2-2\ell_{n,n-2}] \\ 
   & +  \{ib_2-(n-1)a_2-[ib_3-(n-2)a_3]\overline{\alpha}_{n}\alpha_{n-1}\}\overline{\alpha}_{n}\alpha_{n-1}
 \\ 
 & +  [ib_3-(n-2)a_3](1 - |\alpha_{n-1}|^2) \overline{\alpha}_{n}\alpha_{n-2}+[ib_3-(n-1)a_3](1 - |\alpha_{n}|^2) \overline{\alpha}_{n+1}\alpha_{n-1}, 
\end{align*}
and 
\begin{align*}
\s_{n,n+1} = &
\Big\{ (n+1)a_2 -  a_3 \ell_{n+1,n} |\alpha_n|^2-\overline{\ell}_{n+1,n}[a_1+a_{0}\overline{\ell}_{n+1,n}] \\
& +2a_{0}[\overline{\ell}_{n+2,n+1}\overline{\ell}_{n+1,n}-\overline{\ell}_{n+2,n}]  \Big\}  \times \frac{1}{1 - |\alpha_n|^2}  \\
& - (ib_3+a_3) \overline{\alpha}_{n+1} \alpha_n  +(ib_2+a_2) -(ib_0+na_0)\alpha_{n+1}\overline{\alpha}_{n-1}, 
\end{align*}
where $\ell_{i,i-1}$, $i=n+2,n+1,n$, and $\ell_{j,j-2}$, $j=n+2,n$, are as in \eqref{monic-OPUC}.
\end{corollary}

\begin{proof}
From Theorem \ref{Thm1-main}, it follows that  
\begin{align*}
\langle A\Phi_n',\Phi_n\rangle= \s_{n,n} \langle \Phi_{n}, \Phi_{n} \rangle + \p_{n,n}\langle z\Phi_{n}^{\ast},\Phi_n \rangle +\R_{n,n}\langle \Phi_{n}^{\ast},\Phi_n \rangle. 
\end{align*}

Notice that 
\begin{align*}
\langle A\Phi_n',\Phi_n\rangle=
na_{3}\langle z^{n+2},\Phi_n\rangle & +
[a_3(n-1)\ell_{n,n-1}+na_2]\langle z^{n+1},\Phi_n\rangle \\
& +[a_{3}(n-2)\ell_{n,n-2}+a_2(n-1)\ell_{n,n-1}+na_1]\langle z^n,\Phi_n\rangle,
\end{align*}
alongside the relations $\langle \Phi_{n}^{\ast},\Phi_n \rangle=-\alpha_{n-1}\langle \Phi_n,\Phi_n \rangle$, 
 $\langle z\Phi_{n}^{\ast},\Phi_n \rangle =(\alpha_{n-1}\ell_{n+1,n}+\overline{\ell}_{n,1})\langle \Phi_n,\Phi_n\rangle$, $\langle z^{n+2},\Phi_n\rangle =[\ell_{n+2,n+1}\ell_{n+1,n}-{\ell_{n+2,n}}]\langle \Phi_n,\Phi_n\rangle$,
 and $\langle z^{n+1},\Phi_n \rangle=-\ell_{n+1,n}\langle \Phi_n , \Phi_n\rangle$, imply
\begin{eqnarray*}
\s_{n,n} & = & \nonumber
 na_1 - a_2 \ell_{n,n-1} + [ib_2-(n-1)a_2] \overline{\alpha}_{n}\alpha_{n-1}
 \\ \nonumber
&& +a_3 \{ [ \,  n ( \ell_{n+2,n+1} 
- \ell_{n,n-1}) + \ell_{n+1,n}  ] \ell_{n+1,n}  - n\ell_{n+2,n} +(n-2)\ell_{n,n-2} \} \nonumber \\
& & - (ib_3+2a_3) \overline{\alpha}_n 
[\alpha_{n-1}(\ell_{n+1,n} - \ell_{n-1,n-2}) - \alpha_{n-2}] + (ib_3+a_3) \overline{\alpha}_{n+1} 
\alpha_{n-1} (1 - |\alpha_{n}|^2). 
\end{eqnarray*}
Thus the required expression for $\s_{n,n}$ follows from
\begin{align*}
\ell_{n+2,n}-\ell_{n+1,n-1}=(1-|\alpha_n|^2)\overline{\alpha}_{n+1}\alpha_{n-1}+\overline{\alpha}_{n+1}\alpha_{n}\ell_{n+1,n}=\overline{\alpha}_{n+1}\alpha_n\ell_{n,n-1}+\overline{\alpha}_{n+1}\alpha_{n-1}. 
\end{align*}
Similarly, as a consequence of Theorem \ref{Thm1-main}, we get
\begin{equation}\label{Coro1-eq1}
\begin{aligned}
\langle A\Phi_{n}',z^{n+1} \rangle = & \ 
 \s_{n,n+1} \langle  \Phi_{n+1}, \Phi_{n+1} \rangle 
 - \s_{n,n} \overline{\ell}_{n+1,n} \langle \Phi_{n}, \Phi_{n} \rangle  \\
 & + \s_{n,n-1}[\overline{\ell}_{n+1,n}\overline{\ell}_{n,n-1}-\overline{\ell}_{n+1,n-1}] \langle \Phi_{n-1}, \Phi_{n-1} \rangle 
+\R_{n,n}\alpha_{n} \langle \Phi_{n},\Phi_{n} \rangle, 
\end{aligned}
\end{equation}
which, in view of \eqref{Eq-A-phi-1}, becomes
\begin{equation}\label{Coro1-eq2}
\begin{aligned}
\langle A\Phi_{n}',z^{n+1} \rangle = &
\{ 
ib_2+(n+2) a_2
 -[ib_1+(n+2)a_1]\overline{\ell}_{n+1,n}  \\
& + [ib_0+(n+2)a_0](\overline{\ell}_{n+2,n+1}\overline{\ell}_{n+1,n}-\overline{\ell}_{n+2,n})\}\langle \Phi_n,\Phi_n\rangle. 
\end{aligned}
\end{equation}
By using \eqref{Coro1-eq1}, \eqref{Coro1-eq2} and replacing the values of $\s_{n,n}$, $\s_{n,n-1}$ and $\R_{n,n}$ as in Theorem \ref{Thm1-main} we have
\begin{align*}
\s_{n,n+1}  = & 
\Big\{ (n+1)a_2  - a_1 \overline{\ell}_{n+1,n} + (ib_0+(n+2)a_0)[\overline{\ell}_{n+2,n+1} \, \overline{\ell}_{n+1,n}  -
\overline{\ell}_{n+2,n} ] \nonumber  \\
 & - (ib_0+(n+1)a_0) \overline{\ell}_{n+1,n} \,
\overline{\ell}_{n+1,n} + (ib_0+na_0)      \overline{\ell}_{n+1,n-1}  -  a_3 \ell_{n+1,n} |\alpha_n|^2\Big\} \frac{1}{1 - |\alpha_n|^2}  \\
& -(ib_3+a_3) \overline{\alpha}_{n+1} \alpha_n  +(ib_2+a_2). 
\end{align*}
After manipulations similar to the ones used to obtain $\s_{n,n}$ we conclude the result.
\end{proof}

The following result presents difference equations for the associated Verblunsky
coefficients.  

\begin{theorem}
    \label{Coro-2} Under the hypotheses of Theorem \ref{Thm1-main}, the sequence $\{\alpha_n\}_{n\geqslant 0}$ of Verblunsky coefficients associated to the MOPUC satisfy the difference equation
\begin{align*}
[(n-1)\overline{a}_{0}-i\overline{b}_0]  (1-&|\alpha_{n-1}|^2)\alpha_{n-2}+[(n-1)\overline{a}_3 + i\overline{b}_3] (1-|\alpha_n|^2)\alpha_{n+1}    \\
& = 
-\{(n-1)\overline{a}_2+i\overline{b}_2-[i\overline{b}_3+(n-1)\overline{a}_3]\alpha_{n}\overline{\alpha}_{n-1}-2\overline{a}_3\overline{\ell}_{n,n-1}\}\alpha_{n} \\
& \quad - \{(n+1)\overline{a}_1-i\overline{b}_1+[i\overline{b}_0-(n+1)\overline{a}_0]\overline{\alpha}_n{\alpha}_{n-1}-2\overline{a}_{0}\ell_{n,n-1}\}\alpha_{n-1}.
\end{align*}    
\end{theorem}
\begin{proof} By making $z=0$ in the structure relation of Theorem \ref{Thm1-main} and taking its complex conjugate, we get
\begin{equation*}
\overline{a}_{0}(-\overline{\ell}_{n,1})=n\overline{a}_3\alpha_{n+1}+\overline{\s}_{n,n+1}{\alpha}_n+\overline{\s}_{n,n}\alpha_{n-1}+\overline{\s}_{n,n-1}\alpha_{n-2}+\overline{\R}_{n,n}.
\end{equation*}
Finally, we use $-\overline{\ell}_{n,1}=(1-|\alpha_{n-1}|^2)\alpha_{n-2}+\alpha_{n-1}\ell_{n,n-1}$, $\ell_{n,n-1}-\ell_{n+2,n+1}=-(\overline{\alpha}_n\alpha_{n-1}+\overline{\alpha}_{n+1}\alpha_n)$ and the expressions for $\R_{n,n}$ and $\s_{n,k}$ found in Theorem \ref{Thm1-main} for $k=n-1,n,n+1$ to obtain the difference equation. 
\end{proof}

Notice that the restriction of the above theorem to the case where $a_3=b_3=0$ results in the difference equations described in \cite{Bracciali-Rampazzi-SRibeiro-JAT2023}, which, in certain cases, includes a type of discrete Painlevé II. We return to the link between discrete Painlevé II equations and orthogonal polynomials on the unit circle in Subsections \ref{Subsec-generalized-circular-jacobi} and \ref{Sec-example-Bessel-Polynomials}. 

\section{Differential equations for the MOPUC}
\label{Sec3}

In this section we establish first and second order differential equations for the class of orthogonal polynomials $\Phi_n$ on the unit circle under study. These equations follow as consequence of the general structure relation presented in Theorem \ref{Thm1-main}.

\begin{theorem}\label{Thm-main-firstorder-edo} If the weight function $w$ satisfy \eqref{Eq-Tipo-Pearson-1} with $A(z)$ and $B(z)$ as in \eqref{defin-A-B-grau3}, and $A(1)[w(2\pi)-w(0)]=0$, then the semiclassical MOPUC $\Phi_n(z)$ and  $\Phi_n^{\ast}(z)$ satisfy
\begin{eqnarray}
\label{first-derivative-1}
zA(z)\Phi_{n}'(z) &=& U_n(z)\Phi_{n}(z)+V_n(z)\Phi_{n}^{\ast}(z), \\ \label{first-derivative-2}
    -\widehat{A}(z)(\Phi_{n}^{\ast})'(z) &=& W_n(z)\Phi_{n}(z)+Y_n(z)\Phi_{n}^{\ast}(z), 
\end{eqnarray}
with 
\begin{eqnarray*}
\widehat{A}(z)& =&\overline{a}_0z^3+\overline{a}_{1}z^2+\overline{a}_2z+\overline{a}_{3}, \\
U_n(z)&=&na_{3}z^3+(\s_{n,n+1}+na_{3}\overline{\alpha}_{n+1}\alpha_n)z^2+\s_{n,n}z+ib_{0}+na_{0}, \\
V_n(z)&=& [ib_3-(n-2)a_3]\overline{\alpha}_nz^2-(s_{n,n+1}\overline{\alpha}_n+na_{3}\overline{\alpha}_{n+1}-\R_{n,n})z+(ib_{0}+na_{0})\overline{\alpha}_{n-1},  \\
W_n(z)&=&(-i\overline{b}_{0}+n\overline{a}_0)\alpha_{n-1}z^2-(\overline{s}_{n,n+1}\alpha_n+n\overline{a}_{3}\alpha_{n+1}-\overline{\R}_{n,n})z+\overline{\p}_{n,n}-n\overline{a}_3\alpha_n, \\ 
Y_n(z)&=&(-i\overline{b}_{0})z^2+(\overline{\s}_{n,n}-n\overline{a}_{1})z+\overline{\s}_{n,n+1}+n\overline{a}_3\alpha_{n+1}\overline{\alpha}_n-n\overline{a}_2. 
\end{eqnarray*}
Here $\alpha_n$ is the Verblunsky coefficient, and the values $\s_{n,n}$, $\s_{n,n+1}$, $\R_{n,n}$ and $\p_{n,n}$ are as in Theorem \ref{Thm1-main}.
\end{theorem}

\begin{proof} Applying (Szeg\H{o}'s recurrence) \eqref{Szego-recurrence} for $\Phi_{n+2}$, $\Phi_{n+1}$ in Theorem \ref{Thm1-main} results in

\begin{equation}\label{proof-first-derivative-eq-1}
\begin{aligned}
    A(z)\Phi_n'(z) & =[na_3z^2+\s_{n,n+1}z+\s_{n,n}]\Phi_n(z)-na_{3}\overline{\alpha}_{n+1}\Phi_{n+1}^{\ast}(z)
     \\
    &+[-(na_3\overline{\alpha}_n-\p_{n,n})z+\R_{n,n}-\s_{n,n+1}\overline{\alpha}_n]\Phi_{n}^{\ast}(z)
    +\s_{n,n-1}\Phi_{n-1}(z). 
\end{aligned} 
\end{equation}
Now by multiplying \eqref{proof-first-derivative-eq-1} by $z$, \eqref{first-derivative-1} follows from 
$\Phi_{n+1}^{\ast}(z)=\Phi_n^{\ast}(z)-\alpha_{n}z\Phi_{n}(z)$ and \linebreak $\s_{n,n-1}z\Phi_{n-1}(z)=(ib_{0}+na_{0})[\Phi_{n}(z)+\overline{\alpha}_{n-1}\Phi_{n}^{\ast}(z)]$.

As for \eqref{first-derivative-2}, we replace $z$ with $1/\overline{z}$ in \eqref{first-derivative-1} and take the complex conjugate followed by multiplication by $z^{n+3}$ to obtain
\begin{equation*}
\widehat{A}(z)z^{n-1}\overline{\Phi_{n}'(1/\overline{z})}=z^{3}\overline{U_{n}(1/\overline{z})}\Phi_{n}^{\ast}(z)+z^{3}\overline{V_{n}(1/\overline{z})}\Phi_{n}(z). 
\end{equation*}
The result then follows from  $z^{n-1}\overline{\Phi_{n}'(1/\overline{z})}=n\Phi_{n}^{\ast}(z)-z(\Phi_n^{\ast})'(z)$ after some simplifications. 
\end{proof}

The first order differential equations of the Theorem \ref{Thm-main-firstorder-edo} can be presented in a couple of equivalent forms. For instance, the Szeg\H{o}'s recurrences can be used to show that

\begin{corollary}\label{Coro1-edo-first-order} Under the hypotheses of Theorem \ref{Thm-main-firstorder-edo} we have
\begin{eqnarray*}
    zA(z)\Phi_n'(z) & = & [U_n(z)-V_n(z)\alpha_{n-1}]\Phi_n(z)+(1-|\alpha_{n-1}|^2)V_n(z)\Phi_{n-1}^{\ast}(z), \\
    -z\widehat{A}(z)(\Phi_n^{\ast})'(z)& = & W_n(z)\Phi_{n+1}(z)+[W_n(z)\overline{\alpha}_{n}+zY_n(z)]\Phi_n^{\ast}(z).
\end{eqnarray*}  
\end{corollary}

\begin{corollary}\label{Edo-firs-order-Ismailtype}
Under the hypotheses of Theorem \ref{Thm-main-firstorder-edo} we have
\begin{eqnarray*}
   \overline{\alpha}_{n-1}zA(z)\Phi_n'(z) & = & [\overline{\alpha}_{n-1}U_n(z)-V_n(z)]\Phi_n(z)+(1-|\alpha_{n-1}|^2)zV_n(z)\Phi_{n-1}(z), \\
    \overline{\alpha}_{n-1}\widehat{A}(z)(\Phi_n^{\ast})'(z)& = & [Y_n(z)-\overline{\alpha}_{n-1}W_n(z)]\Phi_{n}(z)-(1-|\alpha_{n-1}|^2)zY_n(z)\Phi_{n-1}(z).
\end{eqnarray*}      
\end{corollary}

\begin{remark} The equations of Corollary \ref{Coro1-edo-first-order} appeared in a different context in \cite{branquinho2024} for the particular cases of the weight functions considered in Sections \ref{Example-ranga-WF} and \ref{Sec-example-Bessel-Polynomials}.  
\end{remark}

\begin{remark} When the Verblunsky coefficients are different from zero, Corollary \ref{Edo-firs-order-Ismailtype} is a necessary step to implement the method developed in \cite{ISMAIL-JAT_2001} to obtain differential equations using ladder operators.
\end{remark}

Now that we have the necessary requisites, we present in the next result second order differential equations for the MOPUC. We stress that our approach is rather straightforward and follows from the equations in Theorem \ref{Thm-main-firstorder-edo}.

\begin{theorem} \label{Thm3-second-order-1} Under the hypotheses of Theorem \ref{Thm-main-firstorder-edo}, $\Phi_n(z)$ satisfy
    \begin{align*}
    zA(z&)\Phi_n''(z) +\big\{[zA(z)]'-U_n(z)+\Theta(|V_n|)zA(z)Y_n(z){[\widehat{A}(z)]}^{-1}\big\}\Phi_n'(z)\\&-\big\{{\Theta(|V_n|)[Y_n(z)U_n(z)-V_n(z)W_n(z)]}{[\widehat{A}(z)]}^{-1}+U_n'(z)\big\}\Phi_n(z)=V_n'(z)\Phi_{n}^{\ast}(z), 
    \end{align*}
where $\Theta$ is the unit step function with the convention $\Theta(0)=0$.
 \end{theorem}
\begin{proof}
   If $V_n\equiv 0$, the second order differential equation is obtained from the derivative of \eqref{first-derivative-1}. 
   If $V_n\neq 0$, let 
$\M_n(z)=\exp\left({\int Y_n(z)/\widehat{A}(z)dz}\right)[V_n(z)]^{-1}$.
By multiplying \eqref{first-derivative-1} by $\M_n(z)$ and taking the derivative with respect to  $z$ we obtain
\begin{align*}
zA(z)&\M_n(z)\Phi_n''(z) +\{zA(z)\M'(z)+[(zA(z))'-U_n(z)]\M_n(z)\}\Phi_n'(z) \\
& -\M_n'(z)[U_n(z)\Phi_n(z)+V_n(z)\Phi_n^{\ast}(z)]-\M_{n}(z)[U_{n}'(z)\Phi_n(z)+V_n(z)(\Phi_n^{\ast})'(z)] \\
&=\M_n(z)V_{n}'(z)\Phi_{n}^{\ast}(z).
\end{align*}
The definition of $\M_n$ furnishes
\begin{align*}
\M_n'(z)=\left(\frac{Y_n(z)}{\widehat{A}(z)}-\frac{V_n'(z)}{V_{n}(z)}\right)\M_n(z),
\end{align*}
that can be combined with \eqref{first-derivative-1} to produce
\begin{align*}
zA(z)\Phi_n''(z) &+\big\{[zA(z)]'+{z A(z) Y_n(z)}{[\widehat{A}(z)]^{-1}}-U_n(z)\big\}\Phi_n'(z) \\
& -[U_n(z)\Phi_n(z)+V_n(z)\Phi_n^{\ast}(z)]{Y_n(z)}[\widehat{A}(z)]^{-1} \\
& - U_n'(z)\Phi_n(z)-V_{n}(z)(\Phi_n^{\ast})'(z) \\
& =V_{n}'(z)\Phi_{n}^{\ast}(z). 
\end{align*}
To finish, we use \eqref{first-derivative-2} to obtain that 
\begin{align*}
 [U_n(z)\Phi_n(z)+V_n(z)\Phi_n^{\ast}&(z)]Y_n(z)[\widehat{A}(z)]^{-1}+U_n'(z)\Phi_n(z)+V_{n}(z)(\Phi_n^{\ast})'(z) \\
 & = \big\{ [Y_n(z)U_n(z)-V_n(z)W_n(z)][\widehat{A}(z)]^{-1} + U_{n}'(z)\big\}\Phi_n(z). 
\end{align*} 
\end{proof}
Notice that if $V_n'(z)=0$, the equation of Theorem \ref{Thm3-second-order-1} is manifestly linear. As an immediate consequence of Theorem \ref{Thm3-second-order-1}, if $V'_{n}(z)\neq 0$, equation \eqref{first-derivative-1} gives rise to
\begin{corollary}
\label{coro-linear-second-order-linear-phin}
Under the hypotheses of Theorem \ref{Thm3-second-order-1}, if $V_{n}(z)\neq 0$, then $\Phi_n$ satisfy 
    \begin{align*}          
    z&A(z)V_n(z)\Phi_n''(z) \\
    &+\big
    \{V_n(z)[zA(z)]'-V_n'(z)zA(z)+[Y_n(z)zA(z)-U_n(z)\widehat{A}(z)]{V_n(z)}{[\widehat{A}(z)]^{-1}}\big\}\Phi_n'(z)\\
    & \ \   -\big\{V_n(z)U'_n(z)-V_n'(z)U_n(z)
    +[Y_n(z)U_n(z)-V_n(z)W_n(z)]{V_n(z)}{[\widehat{A}(z)]^{-1}}\big\}\Phi_n(z)=0.
    \end{align*}
\end{corollary}

In a similar fashion, a ``reversed version'' of the differential equation of Theorem \ref{Thm3-second-order-1} can be found for the reversed polynomial $\Phi_n^*$. In fact, we have

\begin{theorem}\label{Thm4-second-order-1} Under the hypotheses of Theorem \ref{Thm-main-firstorder-edo}, we have
\begin{align*}
    \widehat{A}(z)&(\Phi_n^{\ast})''(z) +\big\{(\widehat{A})'(z)+Y_n(z)- \Theta(|W_n|)\widehat{A}(z)U_n(z)[zA(z)]^{-1}\big\}(\Phi_n^{\ast})'(z) \\
    & +\big\{\Theta(|W_n|)[V_n(z)W_n(z)-U_n(z)Y_n(z)][zA(z)]^{-1}+Y_n'(z)\big\}\Phi_n^{\ast}(z) =-W'_n(z)\Phi_{n}(z).
\end{align*}
where $\Theta$ is the unit step function with the convention $\Theta(0)=0$.
\end{theorem}
\begin{proof}
If $W_n\equiv0$, the result follows by taking the derivative of \eqref{first-derivative-2}. If $W_n\neq 0$, let $\mathcal{N}_n(z)=\exp\left({-\int U_n(z)/zA(z)dz}\right)[W_n(z)]^{-1}$. By multiplying \eqref{first-derivative-2} by $\mathcal{N}_n(z)$ and taking the derivative of the outcome we have
\begin{align*}
    \widehat{A}(z)\mathcal{N}_n(z)&(\Phi_n^{\ast})''(z) +    [(\widehat{A})'(z)+Y_n(z)]\mathcal{N}_n(z)
    (\Phi_n^{\ast})' +Y_n'(z)\mathcal{N}_n(z)\Phi_n^{\ast}(z)\\  
    & +
    \mathcal{N}_n'(z)\{\widehat{A}(z)(\Phi_n^{\ast})'(z)+Y_n(z)\Phi_n^{\ast}(z)+W_n(z)\Phi_n(z)\}
    +W_n(z)\mathcal{N}_n(z)\Phi_n'(z) \\
    & =-W_n'(z)\mathcal{N}_n(z)\Phi_n(z). 
\end{align*}
Also, it follows from the definition of $\mathcal{N}_n$ that 
\begin{equation*}
 \mathcal{N}'_n(z)=\left(-\frac{U_n(z)}{zA(z)}-\frac{W_n'(z)}{W_{n}(z)}\right)\mathcal{N}_n(z), 
\end{equation*}
which can be combined with \eqref{first-derivative-2} to produce
\begin{align*}
    \widehat{A}(z)(\Phi_n^{\ast})''(z)& + \{(\widehat{A})'(z)+Y_n(z)-\widehat{A}(z)U_n(z)[zA(z)]^{-1}\}
    (\Phi_n^{\ast})'  \\
    & +\{-Y_n(z)U_n(z)[zA(z)]^{-1}+Y_n'(z)\}\Phi_n^{\ast}(z)\\  
    & +\{-W_{n}(z)U_{n}(z)\Phi_{n}(z)+W_{n}(z)zA(z)\Phi_n'(z)\}[zA(z)]^{-1} \\
    & =-W_n'(z)\Phi_n(z). 
\end{align*}
To conclude, notice that from \eqref{first-derivative-1},
\begin{align*}
[-W_{n}(z)U_{n}(z)\Phi_{n}(z)+W_{n}(z)zA(z)\Phi_n'(z)][zA(z)]^{-1}=V_{n}(z)W_{n}(z)[zA(z)]^{-1}\Phi_{n}^{\ast}(z).
\end{align*}
\end{proof}

Naturally, we also obtain a version of Corollary \ref{coro-linear-second-order-linear-phin} for the reversed polynomial. If $W_n'(z)\neq 0$, \eqref{first-derivative-1} can be used to replace $\Phi_n(z)$ and get
\begin{corollary}
Under the hypotheses of Theorem \ref{Thm4-second-order-1}, if $W_{n}(z)\neq 0$, then $\Phi^*_n$ satisfy   
\begin{align*}
& \big\{
  W_n(z)Y_n'(z)-W_n'(z)Y_n(z)+[V_n(z)W_n(z)-U_n(z)Y_n(z)]{W_n(z)}{[zA(z)]^{-1}}
  \big\}(\Phi_n^{\ast})(z) \\
&+\big\{W_n(z)(\widehat{A})'(z)
  -W_n'(z)\widehat{A}(z)
  +[Y_n(z)zA(z)-\widehat{A}(z)U_n(z)]{W_n(z)}{[zA(z)]^{-1}}\big\}(\Phi_n^{\ast})'(z) \\ 
 & \ \ \ \ +W_n(z)\widehat{A}(z) (\Phi_n^{\ast})''(z)=0. \\ 
\end{align*}
\end{corollary}
\section{Applications}
\label{Sec4}
In this section we present concrete examples of first and second order differential equations for several semiclassical weight functions. Specifically, we consider and generalize some of the weight functions that were characterized in \cite{Bracciali-Rampazzi-SRibeiro-JAT2023} to illustrate applications of Theorems \ref{Thm-main-firstorder-edo}, \ref{Thm3-second-order-1}, and Corollary \ref{coro-linear-second-order-linear-phin}. 

\subsection{Example 1}\label{Example-ranga-WF}
Let us consider first the semiclassical weight function that appeared in \cite{Ranga-PAMS2010}
\begin{equation}\label{medida-ranga}
w(\theta)=e^{-\eta\theta}[\sin^2(\theta/2)]^{\lambda}, \  \lambda>-1/2, \ \ \eta\in\mathbb{R}. 
\end{equation}
Let $b=\lambda+i\eta$. Then $w(\theta)$ satisfy \eqref{Eq-Tipo-Pearson-1} with
    \begin{align}\label{AeB-Ranga2010} A(z) = z-1 \ \mbox{and} \ B(z) = i[(b+1)z+\overline{b}]. 
    \end{align}
The various coefficients from Theorem \ref{Thm1-main} and Corollary \ref{Coro1-formula} read
 \begin{align*}
 \s_{n,n-1}=-(\overline{b}+n)(1-|\alpha_{n-1}|^2), \  \s_{n,n}=n, \ \mbox{and} \ \s_{n,n+1}=\p_{n,n}=
\R_{n,n}=0. 
 \end{align*}
 Structure relations and difference equations for this measure were explored in detail in  \cite{Bracciali-Rampazzi-SRibeiro-JAT2023}. The Verblunsky coefficients are explicitly given as
 \begin{align*}
     \alpha_{n-1}= -\frac{(b)_n}{(\overline{b}+1)_n}, \ \ \ n\geqslant 1.
 \end{align*}

 The next theorem presents the first order differential equations for the associated MOPUC.
 \begin{theorem}\label{Thm-aplic-medidaranga-edo1} Let $\Phi_n$ be the MOPUC corresponding to the weight function \eqref{medida-ranga}. Then
 \begin{align*}
     z(z-1)\Phi_n'(z) & =[nz-(\overline{b}+n)]\Phi_n(z)-(\overline{b}+n)\overline{\alpha}_{n-1}\Phi_n^{\ast}(z),  \\ 
     (1-z)(\Phi_n^{\ast})'(z) &=(b+n)\alpha_{n-1}\Phi_n(z)+b \,\Phi_n^{\ast}(z). 
 \end{align*}  
\end{theorem}

\begin{proof}  If $A$ and  $B$ are as in \eqref{AeB-Ranga2010}, from Theorem \ref{Thm-main-firstorder-edo} we get $\widehat{A}(z)=-z^2(z-1)$, $U_n(z)=nz-(\overline{b}+n)$, $V_n(z)=-(\overline{b}+n)\overline{\alpha}_{n-1}$, $W_n(z)=-(b+n)\alpha_{n-1}z^2$ and $Y_n(z)=-bz^2$. 
\end{proof}

\begin{remark} Let $\tau_n={\Phi_n(1)}/{\Phi_n^{\ast}(1)}$. By taking $z=1$ in the second equation in Theorem \ref{Thm-aplic-medidaranga-edo1}, it follows that $\tau_n=-b[(b+n)\alpha_{n-1}]^{-1}$ and 
    ${(1-z)(\Phi_n^{\ast})'(z)}[{(b+n)\alpha_{n-1}}]^{-1}=\Phi_n(z)-\tau_n\Phi_{n}^{\ast}(z)$.
Since $|\tau_n|=1$, $(1-z)(\Phi_n^{\ast})'(z)$ is a para-orthogonal polynomial. Para-orthogonal polynomials are studied in detail, for instance, in \cite{jones_moment_1989,COSTA201314}, and references therein.
\end{remark}

\begin{corollary}\label{Edo-medida-ranga} 
The polynomials $\Phi_n$ and $\Phi_n^{\ast}$ satisfy
 \begin{align*}
     & z(1-z)\Phi_n''(z)+[(n-b-2)z-\overline{b}-n+1]\Phi_n'(z)+n(1+b)\Phi_n(z)=0, \\
     & z(1-z)(\Phi_n^{\ast})''(z) +[(n-b-1)z-(\overline{b}+n)](\Phi_n^{\ast})'(z)+nb\,\Phi_n^{\ast}(z)=0, 
 \end{align*}
 where $b=\lambda+i\eta$.
\end{corollary}
\begin{proof} Notice that $V_n'(z)=0$ and $Y_{n}(z)U_{n}(z) - V_{n}(z)W_{n}(z)=bn\widehat{A}(z)+z^2(|b|^2-|b+n|^2|\alpha_{n-1}|^2)$.  From Proposition 4.2 and (4.11) of \cite{Bracciali-Rampazzi-SRibeiro-JAT2023}, we have  $|b+n|^2|\alpha_{n-1}|^2-|b|^2=0$ and we obtain the second differential equation for $\Phi_n$ from Theorem \ref{Thm3-second-order-1}. 

The second order differential equation for $\Phi_n^{\ast}$ can be derived from Theorem \ref{Thm4-second-order-1} by redefining $\widehat{A}(z)=(z-1)$, $W_n(z)=(b+n)\alpha_{n-1}$ and $Y_n(z)=b$. This leads to a simplification of common $z^2$ factors in the first order differential equation of Theorem \ref{Thm-aplic-medidaranga-edo1}. Thus the result follows from $|b+n|^2|\alpha_{n-1}|^2-|b|^2=0$. 
\end{proof}

\begin{remark} The differential equations from Corollary \ref{Edo-medida-ranga} can be inferred from the results of \cite{Ranga-PAMS2010}, since $\Phi_n$ in this case is expressed in terms of the Gauss hypergeometric  as
\begin{equation*}
    \Phi_{n}(b;z)=\frac{(b+\overline{b}+1)_n}{(b+1)_n} \,_2F_1(-n,b+1;b+\overline{b}+1;1-z), \ b=\lambda+i\eta.
\end{equation*}
When $\eta=0$,
these coincide with the (monic) circular Jacobi orthogonal polynomials (see Example 1 in \cite{ISMAIL-JAT_2001}). 
\end{remark}

\subsection{Example 2} \label{Subsec-42-example2}

Now we consider the following generalization of the semiclassical weight function of the previous example
\begin{equation}\label{bessel-generalizado-2}
    w(\theta)=e^{t\cos(\theta)}e^{-\eta \theta}[\sin^2(\theta/2)]^\lambda, \ t\in\mathbb{R}, \ \eta\in\mathbb{R} \ \mbox{and} \ \lambda>-1/2.
\end{equation}
If $b=\lambda+i\eta$, this weight function satisfy \eqref{Eq-Tipo-Pearson-1} with 
\begin{equation}\label{A-grau2-B-grau3}
    A(z)=z(z-1), \ \mbox{and} \ B(z)=i\{(t/2)z^3+[b+2-(t/2)]z^2+[\overline{b}-1-(t/2)]z+(t/2)\}.
\end{equation}
Notice that if $\lambda=\eta=0$ we obtain the weight function of modified Bessel polynomials \cite{periwal1990unitary,vanassche2018orthogonal}. If $t\neq 0$ and $\lambda\neq 0$ or $\eta\neq 0$ the weight function belongs to the class $(2,3)$. The coefficients from Theorem \ref{Thm1-main} are
\begin{align*}
 & \s_{n,n-1}=-(t/2)(1-|\alpha_{n-1}|^2), \ \s_{n,n}=-(\overline{b}+n)+(t/2)(1+{\alpha}_n\overline{\alpha}_{n-1}),  \\
 & \R_{n,n}=-(t/2)\overline{\alpha}_{n+1}(1-|\alpha_n|^2)-(b+1-t/2)\overline{\alpha}_n, \  \\
& \s_{n,n+1}=n-(t/2)\overline{\alpha}_n\alpha_{n-1}, \ \mbox{and} \ \p_{n,n}=-(t/2)\overline{\alpha}_n.
 \end{align*}
 From Corollary \ref{Coro1-formula} we also obtain 
 \begin{align*}
\s_{n,n}=&-n-\ell_{n,n-1}-\{b+n+1-(t/2)(1+\overline{\alpha}_n{\alpha}_{n-1})\}\overline{\alpha}_n{\alpha}_{n-1}\\
&- (t/2)[(1-|\alpha_{n-1}|^2)\overline{\alpha}_n\alpha_{n-2}+(1-|\alpha_n|^2)\overline{\alpha}_{n+1}\alpha_{n-1}],
\end{align*}
and
\begin{equation*}
    \s_{n,n+1}=(n+1+\overline{\ell}_{n+1,n})({1-|\alpha_n|^2})^{-1}-b-1+(t/2)[1+\overline{\alpha}_{n+1}\alpha_n+\alpha_{n+1}\overline{\alpha}_{n-1}].
\end{equation*}
The two formulas for $\s_{n,n}$ lead to
\begin{align*}
    \ell_{n,n-1}=& \ \overline{b}+ (b+n+1)\overline{\alpha}_n\alpha_{n-1} +(t/2)[\alpha_{n-1}\overline{\alpha}_n-\overline{\alpha}_{n-1}\alpha_n-1] \nonumber\\
    &-(t/2)\{(1-|\alpha_{n-1}|^2)\overline{\alpha}_n\alpha_{n-2}+(1-|\alpha_{n}|^2)\overline{\alpha}_{n+1}\alpha_{n-1}-(\overline{\alpha}_n\alpha_{n-1})^2\}.
\end{align*}

From Theorem \ref{Coro-2} we obtain the following difference equation.

\begin{theorem}\label{Theo-verb-new-new} If $b=\lambda+i\eta$, the Verblunsky coefficients $\alpha_n$ satisfy 
\begin{align*}
    (\overline{b}&+n+1)\alpha_n-(b+n)\alpha_{n-1} \\
    &=-\frac{t}{2}\big[(1-|\alpha_n|^2)\alpha_{n+1}-(1-|\alpha_{n-1}|^2)\alpha_{n-2}+(1+\overline{\alpha}_n\alpha_{n-1})\alpha_{n-1}-(1+{\alpha}_n\overline{\alpha}_{n-1})\alpha_{n}\big].
\end{align*}

\end{theorem}

\begin{remark}\label{Verblunsky-paper2010}
    Notice that if $t=0$, we recover the weight function \eqref{medida-ranga} from the previous example, that also satisfy \eqref{Eq-Tipo-Pearson-1} with $A(z)=z(z-1)$ and $B(z)=i\{(b+2)z^2+(\overline{b}-1)z\}$. In view of Theorem \ref{Theo-verb-new-new}, the Verblunsky coefficients satisfy $(\overline{b}+n+1)\alpha_n=(b+n)\alpha_{n-1}$.
\end{remark}

Considering $A(z)$ and $B(z)$ as in \eqref{A-grau2-B-grau3} from Theorem \ref{Thm-main-firstorder-edo} we obtain the first order differential equations.

 \begin{theorem}\label{EDO-ordem1-newnew} Let $\Phi_n$ be the MOPUC with respect to the weight function \eqref{bessel-generalizado-2}, $b=\lambda+i\eta$ and $\s_{n,n}=-(\overline{b}+n)+(t/2)(1+{\alpha}_n\overline{\alpha}_{n-1})$. Then 
 \begin{align*}
 z^2(z-1)&\Phi'_n(z) = U_{n}(z)\Phi_n(z)+V_{n}(z)\Phi_n^{\ast}(z), \\
 z(z-1)(&\Phi_n^{\ast})'(z) =W_{n}(z)\Phi_n(z)+Y_{n}(z)\Phi_n^{\ast}(z),
 \end{align*}
 where
 \begin{align*}
 U_n(z)& =[n-(t/2)\overline{\alpha}_n\alpha_{n-1}]z^2+\s_{n,n}z-t/2, \\
 V_n(z)&=-(t/2)\overline{\alpha}_nz^2-[(1-\overline{\s}_{n,n})\overline{\alpha}_n+(t/2)(1-|\alpha_n|^2)\overline{\alpha}_{n+1}]z-(t/2)\overline{\alpha}_{n-1},
 \\
 W_n(z)&=-(t/2)\alpha_{n-1}z^2-[(1-\s_{n,n})\alpha_n+(t/2)(1-|\alpha_n|^2){\alpha}_{n+1}]z-(t/2)\alpha_n, \\
 Y_{n}(z)&=-\{(t/2)z^2-(\overline{\s}_{n,n}+n)z+(t/2)\alpha_n\overline{\alpha}_{n-1}\}.   
 \end{align*}

 \end{theorem}

\begin{remark} If $t=0$ in Theorem \ref{EDO-ordem1-newnew} we recover the results of Theorem \ref{Thm-aplic-medidaranga-edo1}.
\end{remark}

The following result is needed to simplify the coefficients that will appear in Theorem \ref{edo-2-order-bessel-generalized} and also leads to the discrete Painlevé II equation \eqref{tipo-painleve-absolute-value}.

\begin{lemma} \label{lema-weight-newnew}The Verblunsky coefficients satisfy 
\begin{align*}
\big|b+it \Im(\alpha_n\overline{\alpha}_{n-1})\big| = \Big|(b+n)\alpha_{n-1} + (t/2)\left[(1-|\alpha_{n-1}|^2)\alpha_{n-2}+\alpha_n-\overline{\alpha}_n(\alpha_{n-1})^2\right]\Big|.
\end{align*}
\end{lemma}
\begin{proof} By setting  $z=1$ in the first differential equation of Theorem \ref{EDO-ordem1-newnew}, and by defining $\tau_n=\Phi_n(1)/\Phi_n^{\ast}(1)$, since $|\tau_n|=1$, it follows that $\big|U_n(1)\big|^2=\big|V_n(1) \big|^2$, which is equivalent to
\begin{align*}
\big|b+(t/2)(\alpha_n&\overline{\alpha}_{n-1}  -\overline{\alpha}_n\alpha_{n-1})\big|^2 \\
& = \big|(1-\s_{n,n}){\alpha}_n+(t/2)(1-|\alpha_n|^2){\alpha}_{n+1}+(t/2)(\alpha_n+\alpha_{n-1})\big|^2,
\end{align*}
where $\s_{n,n}=-(\overline{b}+n)+(t/2)(1+{\alpha}_n\overline{\alpha}_{n-1})$. The result follows from  Theorem \ref{Theo-verb-new-new},  which guarantees that
\begin{equation*}
    (1-\s_{n,n})\alpha_n+\overline{\s}_{n,n}\alpha_{n-1}=-\frac{t}{2}[(1-|\alpha_n|^2)\alpha_{n+1}-(1-|\alpha_{n-1}|^2)\alpha_{n-2}].
\end{equation*}
\end{proof}

\begin{theorem}\label{edo-2-order-bessel-generalized} Under the hypotheses of Theorem \ref{EDO-ordem1-newnew}, let 
\begin{align*}
&f_{n}=(1-\s_{n,n})\alpha_{n}+(t/2)(1-|\alpha_n|^2)\alpha_{n+1}, \\
& \gamma_{n}=(t/2)(\s_{n,n}+\overline{\s}_{n,n}\overline{\alpha}_n\alpha_{n-1}+\overline{\alpha}_nf_n+\alpha_{n-1}\overline{f}_n).
\end{align*}
The polynomials $\Phi_n$ and $\Phi_n^{\ast}$ satisfy
 \begin{align*}
      z^2(z-1)&\Phi_n''(z)+[(t/2)z^3+(b+3-n-t/2)z^2+(\overline{b}+n-2-t/2)z+t/2]\Phi_n'(z) \\ 
      & -\{(t/2)nz^2+[n(b+2)+\gamma_n-t\overline{\alpha}_n\alpha_{n-1}]z-(t/2)n+\s_{n,n}-\overline{\gamma}_n\}\Phi_n(z) \\
      &=-[t\overline{\alpha}_nz+\overline{f}_n]\Phi_n^{\ast}(z),
 \end{align*}
 and 
 \begin{align*}
     z^2(1&-z)(\Phi_n^{\ast})''(z) \\ 
     &-\{(t/2)z^3+(b+2-n-t/2)z^2+(\overline{b}+n-1-t/2)z+t/2\}(\Phi_n^{\ast})'(z) \\
     & +\{(t/2)(n-2)z^2+[n(b+1)+\gamma_n+\overline{\s}_{n,n}]z-\overline{\gamma}_n-(t/2)n\}\Phi_n^{\ast}(z)\\
     & =(t\alpha_{n-1}z+f_n)\Phi_n(z).
 \end{align*}
\end{theorem}
\begin{proof} Using Lemma \ref{lema-weight-newnew} and performing some manipulations, we obtain
\begin{equation*}
    Y_n(z)U_{n}(z)-V_{n}(z)W_{n}(z)=-z(z-1)\{(t/2)nz^2+(nb+\gamma_n)z-\overline{\gamma}_n-(t/2)n\}.
\end{equation*}
The results follow from Theorems \ref{Thm3-second-order-1} and \ref{Thm4-second-order-1}.
\end{proof}
The linear second order differential equation for $\Phi_n$ is given explicitly in the next theorem.

\begin{theorem}\label{Exemp2-edo-seg-ord} The MOPUC $\Phi_n$ with respect to the weight function \eqref{bessel-generalizado-2} satisfies
\begin{equation*}
z^2(z-1)V_{n}(z)\Phi_n''(z)+p_{4}(z)\Phi_n'(z)-p_{3}(z)\Phi_n(z)=0,
\end{equation*}
where
\begin{align*}
    p_4(z)=V_n(z)[(t/2)z^3&+(b+3-n-t/2)z^2+(\overline{b}+n-2-t/2)z+t/2] \\
    &+z^2(z-1)(t\overline{\alpha}_nz+\overline{f}_n),
\end{align*}
and
\begin{align*}
    p_{3}(z)=V_n(z)[(t/2)nz^2 &+[n(b+2)+\gamma_n-t\overline{\alpha}_n\alpha_{n-1}]z-(t/2)n+\s_{n,n}-\overline{\gamma}_n]\\
    &+(t\overline{\alpha}_nz+\overline{f}_n)U_n(z),
\end{align*}
with $b=\lambda+i\eta$, $V_n$ and $U_n$ as in Theorem \ref{EDO-ordem1-newnew}, and $\gamma_n$ and $f_n$ as in Theorem \ref{edo-2-order-bessel-generalized}.
\end{theorem}

\begin{remark}If $t=0$ in Theorem \ref{Exemp2-edo-seg-ord} we obtain the differential equation previously found in Corollary \ref{Edo-medida-ranga}. This occurs because the weight function \eqref{medida-ranga} belongs to the classes $(1,1)$ and $(2,2)$. 
\end{remark}

\subsubsection{Special case: modified circular Jacobi} \label{Subsec-generalized-circular-jacobi}
When $\eta=0$ we have a  symmetric weight function
\begin{equation}\label{Modified-circular-Jacobi}
w(\theta)=e^{t\cos(\theta)}[\sin^2(\theta/2)]^\lambda.
\end{equation}
For $t=0$, we recover the weight function of the circular Jacobi polynomials \cite{periwal1990unitary,ISMAIL-JAT_2001}. When $\lambda=0$, we have the weight function of the modified Bessel polynomials \cite{ISMAIL-JAT_2001}.

The Verblunsky coefficients are real and from Theorem \ref{Theo-verb-new-new} they satisfy 
\begin{equation}\label{discrete-typepainleve-eq} 
    (\lambda+n+1)\alpha_{n}-(\lambda+n)\alpha_{n-1} =\frac{t}{2}[(1-\alpha_{n-1}^2)(\alpha_{n}+\alpha_{n-2})-(1-\alpha_n^2)(\alpha_{n+1}+\alpha_{n-1})].
\end{equation}

 If $t\neq0$, the same steps used to obtain Lemma \ref{lema-weight-newnew} show that the Verblunsky coefficients 
 are solutions of the discrete Painlevé II equation introduced in \eqref{tipo-painleve-absolute-value}. Observe that in this case the weight function is symmetric and $\tau_n=\Phi_n(1)/\Phi_n^{\ast}(1)=1$.

\begin{remark}
For $t=0$ in \eqref{discrete-typepainleve-eq} the Verblunsky coefficients of the circular Jacobi polynomials satisfy $(\lambda+n+1)\alpha_n=(\lambda+n)\alpha_{n-1}$. 
 When $\lambda=0$ in \eqref{tipo-painleve-absolute-value} we recover the well-known discrete Painlevé II equation associated with the Verblunsky coefficients of modified Bessel polynomials \cite{periwal1990unitary,vanassche2018orthogonal}. 
 \end{remark}

Notice that by taking $b=\lambda$ in Theorem \ref{EDO-ordem1-newnew} and using \eqref{discrete-typepainleve-eq}, we get that $\Phi_n$ and $\Phi_n^{\ast}$ satisfy the first order differential equations 
\begin{equation*}
     z^2(z-1)\Phi'_n(z) = U_{n}(z)\Phi_n(z)+V_{n}(z)\Phi_n^{\ast}(z),
\end{equation*}
where 
\begin{equation}\label{Un-modified-circular-jacob}
U_{n}(z)= [n-(t/2)\alpha_n\alpha_{n-1}]z^2+[(t/2)(1+\alpha_n\alpha_{n-1})-\lambda-n]z-t/2,
\end{equation}
and 
\begin{align} \nonumber
    V_{n}(z)=-(t/2)\alpha_nz^2 &-[(\lambda+n)\alpha_{n-1} 
     +(t/2)(1-\alpha_{n-1}^2)(\alpha_n+\alpha_{n-2}) \\ \label{Vn-modified-circular-jacob}
     & -(t/2)(\alpha_n+\alpha_{n-1})]z  -(t/2)\alpha_{n-1}.
\end{align}
Also,
\begin{equation*}
z(z-1)(\Phi_n^{\ast})'(z) =W_{n}(z)\Phi_n(z)+Y_{n}(z)\Phi_n^{\ast}(z),
\end{equation*}
 with 
$W_{n}(z)=z^2V_{n}(1/z)$ and $Y_{n}(z)=-\{(t/2)z^2+[\lambda-(t/2)(1+\alpha_n\alpha_{n-1})]z+(t/2)\alpha_n\alpha_{n-1}\}$.

\begin{remark}\label{remark-mod-bessel} For $\lambda=0$ and in view of \eqref{tipo-painleve-absolute-value}, it clear that $U_n$ and $V_n$ given in \eqref{Un-modified-circular-jacob} and \eqref{Vn-modified-circular-jacob} are divisible by $z-1$. The same holds for the polynomials $W_n$ and $Y_n$. These simplifications are also applied to the structure relation, because the weight function of modified Bessel polynomials belongs to the class $(1,2)$ and $(2,3)$. 
The simplification of the structure relation, as well as  first and second order differential equations for modified Bessel polynomials will be discussed in more detail in Example 4. 

\end{remark}

\subsection{Example 3}
Next we consider the semiclassical weight function 
\begin{equation}\label{medida-geral}
w(\theta)= e^{-\theta \, \eta}\, [\sin^2(\theta/2)]^{\lambda}[\cos^2(\theta/2)]^{\beta},  
\end{equation}
where $\eta \in \mathbb{R}$,  $\lambda >-1/2$ and $\beta > -1/2$. When $\beta=0$ we recovered the weight function from Example 1, and we note that in this case the differential equation of Theorem \ref{Thm3-second-order-1} is simpler because $V'_n(z)=0$. This, of course, is linked to the fact that this weight function belongs to the class $(1,1)$, which is not the case for $\beta\neq0$. Also, when $\eta=0$ we obtain the weight function of the Jacobi polynomials on the unit circle \cite{badkov1987systems,Magnus2013}.

Let $d=\lambda+\beta+i\eta$. The weight function \eqref{medida-geral} satisfy \eqref{Eq-Tipo-Pearson-1} with
\begin{equation}\label{A-B-grau3}
\begin{aligned}
& A(z) =(z-r)(z^2-1), \\
& B(z)= i\{(d+3)z^3+[2(\lambda-\beta-r)-rd]z^2+[\overline{d}-1-2r(\lambda-\beta)]z-r\overline{d}\},
\end{aligned}
\end{equation}
where  $r\in\mathbb{C}$. This means that the weight function belongs to the class $(3,3)$, and the following result follows from Theorem  \ref{Thm1-main}.

\begin{theorem} Let $d=\lambda+\beta+i\eta$ and $r\in\mathbb{C}$. Then $\Phi_n$ and $\Phi_n^{\ast}$ satisfy, for  $n\geqslant 3$,
\begin{align*}
    (z-r)(z^2-1)&\Phi_{n}'(z)    =  n\Phi_{n+2}(z)-[n(r+\overline{\alpha}_{n+1}\alpha_n)+(d+n)\overline{\alpha}_n\alpha_{n-1}+\ell_{n+1,n}]\Phi_{n+1}(z) \\
 & -\{(\overline{d}+n)(1+r\alpha_n\overline{\alpha}_{n-1})+r[\overline{\ell}_{n+1,n}-2(\lambda-\beta)]\}\Phi_{n}(z) \\
&  +r(\overline{d}+n)(1-|\alpha_{n-1}|^2)\Phi_{n-1}(z) -(d+1)\overline{\alpha}_nz\Phi_{n}^{\ast}(z) \\
&-\{(d+2)\overline{\alpha}_{n+1}(1-|\alpha_n|^2)+[2(\lambda-\beta)-\ell_{n+1,n}-r(d+1)]\overline{\alpha}_n\}\Phi_{n}^{\ast}(z).
\end{align*}
\end{theorem}
 From Theorem \ref{Coro-2} we have the following discrete relation for the Verblunsky coefficients.

\begin{theorem} The Verblunsky coefficients $\alpha_n$  satisfy the difference equation
    \begin{align*} 
\overline{r}(d+n-1)&(1-|\alpha_{n-1}|^2)\alpha_{n-2} +(\overline{d}+n+2)(1-|\alpha_n|^2)\alpha_{n+1} \\
& = \{\overline{r}(n-1)-2(\lambda-\beta-\overline{r})+\overline{r}\overline{d}+(\overline{d}+n)\alpha_{n}\overline{\alpha}_{n-1}+2\overline{\ell}_{n+1,n}\}\alpha_{n} \\
& \ \ \ \ \ \ \ \{n+d-2\overline{r}(\lambda-\beta)+\overline{r}[d+n+1]\overline{\alpha}_n\alpha_{n-1}+2\overline{r}\ell_{n,n-1}\}\alpha_{n-1}. 
\end{align*}
In particular, if $r=0$, 
\begin{align*}
(\overline{d}+n+2)(1-|\alpha_n|^2)\alpha_{n+1}-(d+n)\alpha_{n-1}=\{2(\beta-\lambda)+(\overline{d}+n)\alpha_n\overline{\alpha}_{n-1}+2\overline{\ell}_{n+1,n}\}\alpha_n . 
\end{align*}

\end{theorem}

Although the weight function \eqref{medida-geral} belongs to the class $(3,3)$, it also belongs to the class $(2,2)$, as \eqref{Eq-Tipo-Pearson-1} is satisfied with
\begin{align}\label{A-B-generalizedJac}
A(z) =  z^2 -1 \ \mbox{and} \ B(z) = i[(d +2) \, z^2 + 2(\lambda - \beta)z +\overline{d}]. 
\end{align}
In order to simplify the formulas used to obtain the differential equations, in this section we fix  $A(z)$ and $B(z)$ as in \eqref{A-B-generalizedJac}. Then from Theorem \ref{Thm1-main} and Corollary \ref{Coro1-formula} we get 
\begin{align*}
    & \s_{n,n-1}=-(\overline{d}+n)(1-|\alpha_{n-1}|^2), \ \s_{n,n}=2(\beta-\lambda)+\overline{\ell}_{n+1,n}+(\overline{d}+n)\alpha_{n}\overline{\alpha}_{n-1},  \\
    & \s_{n,n+1}=n,  \ \R_{n,n}=-(d+1)\overline{\alpha}_{n}, \ \p_{n,n}=0, \ \s_{n,n}=-\ell_{n+1,n}-(d+n)\overline{\alpha}_{n}\alpha_{n-1}. 
\end{align*}

\begin{theorem}\label{Firs-OrderDer-Generalweight} Let $\Phi_n$ be the MOPUC with respect to the weight function \eqref{medida-geral}. Then
    \begin{align*}
z(z^2-1)\Phi_n'(z)= &\{nz^2-[\ell_{n+1,n}+(d+n)\overline{\alpha}_n\alpha_{n-1}]z -(\overline{d}+n)\}\Phi_n(z) \\
&-[(d+n+1)\overline{\alpha}_nz+(\overline{d}+n)\overline{\alpha}_{n-1}]\Phi_n^{\ast}(z),
     \end{align*}
and
    \begin{align*}
     (1-z^2)(\Phi_n^{\ast})'(z)= & [(d+n)\alpha_{n-1}z+(\overline{d}+n+1)\alpha_n ]\Phi_n(z) \\
     & +\{dz+\overline{\ell}_{n+1,n}+(\overline{d}+n)\alpha_n\overline{\alpha}_{n-1}\}\Phi_n^{\ast}(z). 
 \end{align*}  
 \end{theorem}

\begin{proof} If $A(z)$ and  $B(z)$ are as in \eqref{A-B-generalizedJac}, from Theorem \ref{Thm-main-firstorder-edo} we get  $\widehat{A}(z)=z(1-z^2)$, $U_n(z)=nz^2-[\ell_{n+1,n}+(d+n)\overline{\alpha}_n\alpha_{n-1}]z-(\overline{d}+n)$, $V_n(z)=-(d+n+1)\overline{\alpha}_n z-(\overline{d}+n)\overline{\alpha}_{n-1}$, $W_n(z)=-z[(d+n)\alpha_{n-1}z+(\overline{d}+n+1)\alpha_n]$ and $Y_n(z)=-z[dz+\overline{\ell}_{n+1,n}+(\overline{d}+n)\alpha_n\overline{\alpha}_{n-1}]$. 
\end{proof}

We need the next technical result to simplify the coefficients appearing in Theorem \ref{Thm3-second-order-1}.

\begin{lemma}\label{value-gamma-geneweight} The coefficient  $\ell_{n,n-1}$ satisfies
\begin{align*}
    (\overline{d}+n)\overline{\ell}_{n,n-1}+(d+n){\ell}_{n,n-1}+2n(\beta-\lambda)=0, \ \ n\geq 1. 
\end{align*}
\end{lemma}
\begin{proof} By making  $z=1$ and $z=-1$ in the first differential equation of Theorem \ref{Firs-OrderDer-Generalweight} and using that $|\Phi_n(1)/\Phi_n^{\ast}(1)|=1$, we obtain
\begin{align}\label{formula1-GenerWeight}
&|(\overline{d}+n+1)\alpha_n+(d+n)\alpha_{n-1}|^2=|\ell_{n+1,n}+(d+n)\overline{\alpha}_n\alpha_{n-1}+\overline{d}|^2,  \\ \label{formula2-GenerWeight}
&|(\overline{d}+n+1)\alpha_n-(d+n)\alpha_{n-1}|^2=|\ell_{n+1,n}+(d+n)\overline{\alpha}_n\alpha_{n-1}-\overline{d}|^2 .   
\end{align}
    The identity \eqref{formula1-GenerWeight} can be written as 
\begin{align*}
|d+n+1|^2|\alpha_n|^2 & +|d+n|^2|\alpha_{n-1}|^2-|\ell_{n+1,n}+(d+n)\overline{\alpha}_n\alpha_{n-1}|^2-|d|^2 \\
& =d\ell_{n+1,n}+\overline{d}\,\overline{\ell}_{n+1,n}-(n+1)[(\overline{d}+n)\alpha_n\overline{\alpha}_{n-1}+(d+n)\overline{\alpha}_n\alpha_{n-1}]. 
\end{align*}
Similarly, we have from \eqref{formula2-GenerWeight} that

\begin{align*}
|d+n+1|^2|\alpha_n|^2 & +|d+n|^2|\alpha_{n-1}|^2-|\ell_{n+1,n}+(d+n)\overline{\alpha}_n\alpha_{n-1}|^2-|d|^2 \\
& =-d\ell_{n+1,n}-\overline{d}\,\overline{\ell}_{n+1,n}+(n+1)[(\overline{d}+n)\alpha_n\overline{\alpha}_{n-1}+(d+n)\overline{\alpha}_n\alpha_{n-1}].
\end{align*}
Hence, we see that 
\begin{equation}\label{formula1-need}
d\ell_{n+1,n}+\overline{d}\,\overline{\ell}_{n+1,n}-(n+1)[(\overline{d}+n)\alpha_n\overline{\alpha}_{n-1}+(d+n)\overline{\alpha}_n\alpha_{n-1}]=0. 
\end{equation}
Therefore, the required equation is a consequence of the two formulas for $\s_{n,n}$, that lead to 
\begin{equation*}
    (\overline{d}+n)\alpha_n\overline{\alpha}_{n-1}+(d+n)\overline{\alpha}_n\alpha_{n-1}=-2(\beta-\lambda)-\overline{\ell}_{n+1,n}-\ell_{n+1,n}. 
\end{equation*}
\end{proof}
\begin{corollary}\label{Thm-EDO-second-order-general-weight1}
The polynomials $\Phi_n$ and $\Phi_n^{\ast}$ satisfy
 \begin{align*}
      z(z^2-1)&\Phi_n''(z)+[(d+3-n)z^2+2(\lambda-\beta)z+\overline{d}+n-1]\Phi_n'(z) \\ 
      & -[n(d+2)z-(d+n+1)\ell_{n+1,n}-2n(\beta-\lambda)]\Phi_n(z)=- (d+n+1)\overline{\alpha}_n\Phi_n^{\ast}(z),
 \end{align*}
 and
 \begin{align*}
     z(1-z^2)(\Phi_n^{\ast})''&(z)- [(d+3-n)z^2+2(\lambda-\beta)z+\overline{d}+n-1](\Phi_n^{\ast})'(z) \\
     & +[d(n-2)z-2n(\beta-\lambda)-(d+n)\ell_{n,n-1}-\overline{\ell}_{n+1,n}-(\overline{d}+n)\alpha_{n}\overline{\alpha}_{n-1}]\Phi_n^{\ast}(z) \\
     & =[2(d+n)\alpha_{n-1}z+(\overline{d}+n+1)\alpha_n]\Phi_n(z).
 \end{align*}
\end{corollary}
\begin{proof}Notice that using \eqref{formula1-need} we obtain
 $Y_n(z)U_n(z)-V_n(z)W_n(z) =-z\{dnz^3-[(d+n)\ell_{n,n-1}+2n(\beta-\lambda)]z^2-dnz-(\overline{d}+n)\overline{\ell}_{n,n-1}\}$, 
or, equivalently
\begin{align*}
    Y_n(z)U_n(z)-V_n(z)W_n(z)
     = & -z\big\{[dnz-(d+n)\ell_{n,n-1}-2n(\beta-\lambda)](z^2-1) \\
    &-[(\overline{d}+n)\overline{\ell}_{n,n-1}+(d+n)\ell_{n,n-1}-2n(\beta-\lambda)]\big\}. 
\end{align*}
The claim follows from Lemma \ref{value-gamma-geneweight} and Theorems \ref{Thm3-second-order-1},  \ref{Thm4-second-order-1}.
\end{proof}
\begin{corollary}\label{Edo-linear-segundaordem-MG} Let $\Phi_n$ be the MOPUC with respect the weight function \eqref{medida-geral}. Then
\begin{align*}
 z(z+\varepsilon_n)(z^2-1) \Phi_n''(z)+p_{3}(z)\Phi_n'(z)-p_{2}(z)\Phi_n(z)=0, 
\end{align*}
where  the sequence $\varepsilon_n$ is given by $(d+n+1)\overline{\alpha}_n\varepsilon_{n}=(\overline{d}+n)\overline{\alpha}_{n-1}$, and
\begin{align*}
p_{3}(z)=&[(d+3-n)z^2+2(\lambda-\beta)z+\overline{d}+n-1](z+\varepsilon_n)-z(z^2-1)],  \\
p_{2}(z)=&[n(d+2)z-(d+n+1)\ell_{n+1,n}-2n(\beta-\lambda)](z+\varepsilon_n)-nz^2 \\
&+[\ell_{n+1,n}+(d+n)\overline{\alpha}_n\alpha_{n-1}]z+\overline{d}+n. 
\end{align*}
\end{corollary}

\begin{remark}
If $\beta=0$, then $d=b=\lambda+i\eta$ and $(\overline{b}+n+1)\alpha_{n}=(b+n)\alpha_{n-1}$ see \cite{Ranga-PAMS2010,Bracciali-Rampazzi-SRibeiro-JAT2023}. In this case, $(b+n)\ell_{n,n-1}=n\overline{b}$, $\varepsilon_{n}=1$, $p_{3}(z)=[(b+2-n)z+\overline{b}+n-1](z+1)^2$ and $p_{2}(z)=n(1+b)(z+1)^2$ and we recover the differential equation of Corollary \ref{Edo-medida-ranga}. 
\end{remark}

\subsubsection{Special case: Jacobi polynomials on the unit circle}\label{Subsec-JP}

Let us consider the semiclassical weight function in \eqref{medida-geral} with $\eta=0$, that is 
\begin{equation}\label{jacobi-unit-circle}
w(\theta)= [\sin^2(\theta/2)]^{\lambda}[\cos^2(\theta/2)]^{\beta}, \  \lambda>-1/2, \beta>-1/2. 
\end{equation}
This weight function produce the \textit{Jacobi polynomials on the unit circle} \cite{badkov1987systems}, that are given explicitly as 
\begin{equation*}
    \Phi_{2n-1}^{(\lambda,\beta)}(z)=(2z)^{n-1
}\left[P_{n}^{(\lambda-1/2,\beta-1/2)}\left(\frac{z^2+1}{2z}\right)+\frac{z^2-1}{2z}P_{n-1}^{(\lambda+1/2,\beta+1/2)}\left(\frac{z^2+1}{2z}\right)\right],
\end{equation*}
and
\begin{equation*}
    \Phi_{2n}^{(\lambda,\beta)}(z)=(2z)^{n
}\left[h_n\,P_{n}^{(\lambda-1/2,\beta-1/2)}\left(\frac{z^2+1}{2z}\right)+g_n\,\frac{z^2-1}{2z}P_{n-1}^{(\lambda+1/2,\beta+1/2)}\left(\frac{z^2+1}{2z}\right)\right],
\end{equation*}
where $h_n=(\lambda+\beta+n)/(\lambda+\beta+2n)$, 
$g_n=n/(\lambda+\beta+2n)$, and $P_{n}^{(a,b)}$ is the monic Jacobi polynomial of degree $n$. The Verblunsky coefficients $\alpha_n$ are real \cite{badkov1987systems}:
\begin{align*}
    \alpha_{n-1}=-\Phi_{n}(0)=-\frac{\lambda+(-1)^{n}\beta}{n+\lambda+\beta}.
\end{align*}
Notice that $\ell_{n,n-1}$ is also real and from Lemma \ref{value-gamma-geneweight} we have $(\lambda+\beta+n)\ell_{n,n-1}=-n(\beta-\lambda)$  and  $\ell_{n+1,n}+(\lambda+\beta+n)\alpha_{n}\alpha_{n-1}=\lambda-\beta$. 

By setting $\eta=0$ in Theorem \ref{Firs-OrderDer-Generalweight}, after some simplifications, we obtain the first order differential equations
\begin{align*} 
    z(z^2-1)\Phi_{n}'(z)=&[nz^2+(\beta-\lambda)z-\lambda-\beta-n]\Phi_n(z) \\
    & +\{[\lambda-(-1)^{n}\beta]z+\lambda+(-1)^n\beta\}\Phi_{n}^{\ast}(z),
\end{align*}
and
\begin{align*}
    (z^2-1)(\Phi_n^{\ast})'(z)=-&\{[\lambda+(-1)^n\beta]z+\lambda-(-1)^n\beta\}\Phi_{n}(z)\\
    &+[(\lambda+\beta)z+\lambda-\beta]\Phi_{n}^{\ast}(z). 
\end{align*}

Finally, from Corollary \ref{Thm-EDO-second-order-general-weight1} we obtain the second order differential equation
\begin{align*}
    z(z^2-1&)\Phi_n''(z)+\{(\lambda+\beta+3-n)z^2+2(\lambda-\beta)z+\lambda+\beta+n-1\}(\Phi_n)'(z) \\
    &-\{n(\lambda+\beta+2)z-(\beta-\lambda)(n-1)\}\Phi_{n}(z)=\{\lambda-(-1)^{n}\beta\}\Phi^{\ast}_n(z), 
\end{align*}
and if $\lambda-(-1)^n\beta\neq 0$, with $\varepsilon_{n}=[\lambda+(-1)^n\beta]/[\lambda-(-1)^n\beta]$, Corollary \ref{Edo-linear-segundaordem-MG} gives rise to 
\begin{equation}\label{eqsecond-jacobi-unit-circle}
    z(z^2-1) (z+\varepsilon_{n})\Phi_n''(z)+p_3(z)\Phi_n'(z)-p_{2}(z)\Phi_n(z)=0, 
\end{equation}
where the polynomials $p_2$ and $p_3$ read, in this particular case,
\begin{align*}
p_{3}(z)= &(\lambda+\beta+2-n)z^3+[2(\lambda-\beta)+(\lambda+\beta+3-n)\varepsilon_n]z^2 \\
& +[\lambda+\beta+n+2(\lambda-\beta)\varepsilon_{n}]z+(\lambda+\beta+n-1)\varepsilon_n,  \\
p_{2}(z)= &n(\lambda+\beta+1)z^2+n[(\lambda+\beta+2)\varepsilon_n+\lambda-\beta]z \\
    &+(n-1)(\lambda-\beta)\varepsilon_n
    +\lambda+\beta+n. 
\end{align*}

\begin{remark}
The linear second order differential equation for the Jacobi polynomials on the unit circle, \eqref{eqsecond-jacobi-unit-circle}, was deduced in \cite{ISMAIL-JAT_2001} using a different approach. If $\beta=0$, we obtain the second order differential equation for circular Jacobi polynomials \cite{ISMAIL-JAT_2001}.
\end{remark}

\subsection{Example 4}
As another example, we consider the semiclassical weight function 
\begin{equation}\label{medida-besselgeneralized}
w(\theta)=e^{2|u|\sin(\theta+\arg(u))}, \quad \theta \in [0,2\pi], \  u\in\mathbb{C}, 
\end{equation}
that satisfies \eqref{Eq-Tipo-Pearson-1} with
\begin{align}\label{AeB-besselG} A(z) = z \ \mbox{and} \ B(z) = uz^2+iz+\overline{u}.  
\end{align}
This measure appeared in \cite{marcellan2017sobolev} in the context of coherent pairs of measures. If $u=it/2$ and $t\in\mathbb{R}$, we obtain the weight function of the \textit{modified Bessel polynomials}, $w(\theta)=e^{t\cos(\theta)}$ \cite{ismail2005classical} already introduced in Example 2 (See Remark \ref{remark-mod-bessel}).
The coefficients from Theorem \ref{Thm1-main} and Corollary \ref{Coro1-formula} are
\begin{align*}
 & \s_{n,n-1}=i\,\overline{u}(1-|\alpha_{n-1}|^2), \ \s_{n,n}=n-i\,\overline{u}{\alpha}_n\overline{\alpha}_{n-1},  \ 
\R_{n,n}=iu\overline{\alpha}_n, \ \p_{n,n}=\s_{n,n+1}=0 \\
& \s_{n,n}=n+iu\overline{\alpha}_n\alpha_{n-1} \ \mbox{and} \ \s_{n,n+1}=-\frac{\overline{\ell}_{n+1,n}}{1-|\alpha_{n}|^2}+iu-i\overline{u}\alpha_{n+1}\overline{\alpha}_{n-1}. 
 \end{align*}
 The structure relation and the discrete equation for the Verblunsky coefficients were recently explored for this measure in \cite{Bracciali-Rampazzi-SRibeiro-JAT2023}. 

 \begin{theorem}\label{BesselG-edo-ordem1} Let $\Phi_n$ be the MOPUC with respect to the weight function \eqref{medida-besselgeneralized}. Then
 \begin{align*}
     z^2\Phi_n'(z) &=[(n-i\,\overline{u}\alpha_n\overline{\alpha}_{n-1})z+i\,\overline{u}]\Phi_n(z)+[i\,u\overline{\alpha}_nz+i\,\overline{u}\,\overline{\alpha}_{n-1}]\Phi_n^{\ast}(z), \\ z(\Phi_n^{\ast})'(z) &=i[u\alpha_{n-1}z+\overline{u}\alpha_n]\Phi_n(z)+iu[z-\overline{\alpha}_n{\alpha}_{n-1}]\Phi_n^{\ast}(z) .
 \end{align*} 
 \end{theorem}
\begin{proof}
If $A(z)$ and  $B(z)$ are as in \eqref{AeB-besselG}, from Theorem \ref{Thm-main-firstorder-edo} we get $\widehat{A}(z)=z^2$, $U_n(z)=(n-i\,\overline{u}\alpha_n\overline{\alpha}_{n-1})z+i\,\overline{u}$, $V_n(z)=iu\overline{\alpha}_nz+i\,\overline{u}\,\overline{\alpha}_{n-1}$, $W_n(z)=-z(iu\alpha_{n-1}z+i\,\overline{u}\alpha_n)$ and $Y_n(z)=-z(iuz-iu\overline{\alpha}_n\alpha_{n-1})$. 
\end{proof}

\begin{corollary}\label{corollary-GBessel-2} The polynomials $\Phi_n$ and $\Phi_n^{\ast}$ satisfy 
 \begin{align*}
      z^2\Phi_n''(z)&+[-iuz^2+(2-n)z-i\overline{u}]\Phi_n'(z)-[n(1-iuz)+|u|^2(1-|\alpha_n|^2)(1-|\alpha_{n-1}|^2)\\&+i(nu\overline{\alpha}_n\alpha_{n-1}-\overline{u}\alpha_n\overline{\alpha}_{n-1})]\Phi_n(z)=iu\overline{\alpha}_n\Phi_n^*(z), 
\end{align*}
and
\begin{align*}
z^2(\Phi_n^{\ast})''(z) &+[-iuz^2+(2-n)z-i\overline{u}](\Phi_n^{\ast})'(z)+[iu(n-2)z-|u|^2(1-|\alpha_n|^2)(1-|\alpha_{n-1}|^2) \\ & +iu\overline{\alpha}_n\alpha_{n-1}(1-n)]\Phi_n^{\ast}(z)=[2iu\alpha_{n-1}z+i\overline{u}\alpha_n]\Phi_n(z). 
 \end{align*}
\end{corollary}
\begin{proof} Notice that
\begin{align*}
    Y_n(z)U_n(z)-V_n(z)W_n(z) = -z&\big\{[iun+|u|^2\alpha_{n}\overline{\alpha}_{n-1}+(u)^2\overline{\alpha}_n\alpha_{n-1}]z^2 \\
    & -[iun\overline{\alpha}_n\alpha_{n-1}+|u|^{2}(1-|\alpha_n|^2)(1-|\alpha_{n-1}|^2)]z \\
    & +|u|^2\overline{\alpha}_{n}{\alpha}_{n-1}+(\overline{u})^2{\alpha}_n\overline{\alpha}_{n-1}
    \big\}.
\end{align*}
We simplify the expression above by comparing the two formulas for $\s_{n,n}$: $|u|^2\alpha_{n}\overline{\alpha}_{n-1}+(u)^2\overline{\alpha}_n\alpha_{n-1}=0.$ We then conclude the proof  using Theorems \ref{Thm3-second-order-1} and  \ref{Thm4-second-order-1}.
\end{proof}

\subsubsection{Special case: modified Bessel polynomials}
\label{Sec-example-Bessel-Polynomials}
Let $u=it/2$ in the semiclassical weight function \eqref{medida-besselgeneralized}, that is, 
\begin{equation*}
w(\theta)=e^{t\cos(\theta)},
\end{equation*}
which is the weight function of the modified Bessel polynomials. 
The Verblunsky coefficients satisfy the discrete Painlevé II equation, given in \eqref{tipo-painleve-absolute-value} for $\lambda=0$ \cite{periwal1990unitary,vanassche2018orthogonal}. 
Let $\Phi_n$ be the modified Bessel polynomial of degree $n$, then from Theorem \ref{BesselG-edo-ordem1} or Theorem \ref{edo-2-order-bessel-generalized}, we have
    \begin{align*}
z^2\Phi_n'(z)&=(t/2)[(2nt^{-1}-\alpha_n\alpha_{n-1})z+1]\Phi_n(z)-(t/2)[{\alpha}_nz-{\alpha}_{n-1}]\Phi_n^{\ast}(z),  \\ -z(\Phi_n^{\ast})'(z)&=(t/2)[\alpha_{n-1}z-\alpha_n]\Phi_n(z)+(t/2)[z-{\alpha}_n{\alpha}_{n-1}]\Phi_n^{\ast}(z).
 \end{align*}

Equivalent form of these equations appeared in \cite{branquinho2024} in a different context and the connection between the results follows from Corollary \ref{Coro1-edo-first-order}.  As for the second order equations, Corollary \ref{corollary-GBessel-2} implies 
   \begin{align*}
        z^2\Phi_n''(z)  &+\big[(t/2)z^2+(2-n)z-(t/2)\big]\Phi_n'(z) \\ \nonumber
         &-\big[n+n(t/2)z+(t/2)^2(1-\alpha_n^2)(1-\alpha_{n-1}^2)-(t/2)\alpha_n\alpha_{n-1}(n+1)\big]\Phi_n(z) \\
         &=-(t/2)\alpha_n\Phi_n^*(z),
    \end{align*}
    and
    \begin{align*}
         z^2(\Phi_n^{\ast})''(z) & +\big[(t/2)z^2+(2-n)z-(t/2)\big](\Phi_n^{\ast})'(z) \\ 
         &- \big[(t/2)(n-2)z+(t/2)^2(1-\alpha_n^2)(1-\alpha_{n-1}^2)+(t/2)\alpha_n\alpha_{n-1}(1-n)\big]\Phi_n^{\ast}(z)  \\
         & =[-t\alpha_{n-1}z+(t/2)\alpha_n]\Phi_n(z). 
   \end{align*}
By using the discrete Painvelé II equation for the Verblunsky coefficients, it is possible to show that these differential equations are equivalent to the ones obtained in \cite{branquinho2024}.

\subsection{Example 5}
As a final example, let us consider the semiclassical weight function 
\begin{align}\label{Exemplo3-JAT-2013}
    w(\theta)=e^{2\Re(u/\overline{r})\arg(1-re^{-i\theta})}|e^{i\theta}-r|^{-2\Im(u/\overline{r})}. 
\end{align}
where $u,r\in \mathbb{C}$, $|r|\neq 1$, $r\neq 0$. This measure was recently explored in \cite{marcellan2017sobolev} in the context of coherent pairs of measures. This weight function  satisfies \eqref{Eq-Tipo-Pearson-1} with 
\begin{align}
&A(z) =  (z-r)\left(z-\frac{1}{\overline{r}}\right), \notag  \\ 
&  B(z) = \left(\frac{2i\overline{r}-{u}}{\overline{r}}\right)z^2+\left[\frac{2\Re(ur)-(1+|r|^2)i}{\overline{r}}\right]z-\frac{\overline{u}}{\overline{r}}. \label{A-B-ultimo-example}
\end{align}
Therefore, the coefficients from Theorem \ref{Thm1-main} and Corollary \ref{Coro1-formula} are
\begin{align*}
    &\overline{r}\s_{n,n-1}=(nr-i\overline{u})(1+|\alpha_{n-1}|^2), \ \s_{n,n+1}=n, \ \overline{r}\R_{n,n}=-(\overline{r}+iu)\overline{\alpha}_n, \ \p_{n,n}=0,  \\
    &\overline{r}\s_{n,n}=[2i\Re(ur)-n(1+|r|^2)-r(\overline{\ell}_{n+1,n}+n\alpha_n\overline{\alpha}_{n-1})+i\overline{u}\alpha_n\overline{\alpha}_{n-1}], \\
    & \overline{r}\s_{n,n}=-n(1+|r|^2)-\overline{r}\ell_{n+1,n}-(\overline{r}n+iu)\overline{\alpha}_n\alpha_{n-1}. 
\end{align*}

\begin{theorem}\label{Edo-fistorder-Exemple3-JAT-2023} Let $\Phi_{n}$ be the MOPUC with respect to the weight function \eqref{Exemplo3-JAT-2013}. Then
\begin{align*}
    z(z-r)(\overline{r}z-1)&\Phi_n'(z) = 
     \{-[\overline{r}(n+1)+iu]\overline{\alpha}_nz+(rn+\overline{iu})\overline{\alpha}_{n-1}\}\Phi^{\ast}_{n}(z)\\
    &+\{\overline{r}nz^2-[n(1+|r|^2)+\overline{r}\ell_{n+1,n}+(\overline{r}n+iu)\overline{\alpha}_{n}\alpha_{n-1}]z+rn+\overline{iu}\}\Phi_n(z), \\
\end{align*} 
and 
\begin{align*}
    (z-r)(1-\overline{r}z)(\Phi_n^{\ast})'(z)  =\{(\overline{r}n&+iu)\alpha_{n-1}z-[r(n+1)+\overline{iu}]\alpha_{n}\}\Phi_n(z) \\
    & +\{iuz - r\overline{\ell}_{n+1,n}-(rn+\overline{iu})\alpha_n\overline{\alpha}_{n-1}\}\Phi_{n}^{\ast}(z). 
\end{align*}
\end{theorem}
\begin{proof} If $A(z)$ and $B(z)$ are as in \eqref{A-B-ultimo-example}, from Theorem \ref{Thm-main-firstorder-edo} we obtain 
$r\widehat{A}(z)=z(z-r)(\overline{r}z-1)$,
$\overline{r}U_n(z)=\overline{r}nz^2-[n(1+|r|^2)+\overline{r}\ell_{n+1,n}+(\overline{r}n+iu)\overline{\alpha}_n\alpha_{n-1}]z+\overline{iu}+{rn}$, 
$\overline{r}V_n(z)=-[\overline{r}(n+1)+iu]\overline{\alpha}_nz+(rn+\overline{iu})\overline{\alpha}_{n-1}$,
$rW_n(z)=z\{(\overline{r}n+iu)\alpha_{n-1}z-[r(n+1)+\overline{iu}]\alpha_{n}\}$ and
$rY_n(z)=z[iuz-r\overline{\ell}_{n+1,n}-(rn+\overline{iu})\alpha_n\overline{\alpha}_{n-1}]$.  
\end{proof}

Before we proceed to the second order differential equations, the following technical result is useful to express the coefficients of the equation in a more compact form.

\begin{lemma}\label{identi-exem3-jaat2023}The following identities hold,
\begin{enumerate}
  \renewcommand{\labelenumi}{\textbf{(\alph{enumi})}}
    \item $\Im[\overline{r}\ell_{n,n-1}+[(n+1)\overline{r}+iu]\overline{\alpha}_{n}\alpha_{n-1}]=-\Re(ur)$,
    \item $\Im[(n\overline{r}+iu)\ell_{n,n-1}]=-n\Re(ur)$,
    \item$n\Im[((n+1)r+\overline{iu})\overline{\ell}_{n+1,n}]-(n+1)\Im[(nr+\overline{iu})\overline{\ell}_{n,n-1}]=0$. 
\end{enumerate}
\end{lemma} 

\begin{proof}Item (a) follows by comparing the two expressions for $\s_{n,n}$. 
Consider the difference equation (see, p.24 \cite{Bracciali-Rampazzi-SRibeiro-JAT2023}), for $n \geqslant 2$,
\begin{equation} \label{diference_equ_Ex3_JAT-2023_1}
[(n+1) r +  \overline{iu}]\alpha_n +[(n-1) \overline{r} + ui]  \alpha_{n-2} = \left[n(|r|^2 + 1) + 2\Re( r \overline{\ell}_{n,n-1})\right] \frac{\alpha_{n-1}}{1-|\alpha_{n-1}|^2}. 
\end{equation}
We multiply \eqref{diference_equ_Ex3_JAT-2023_1} by $\overline{\alpha}_{n-1}$ and simplify the outcome to get
\begin{align*}
\Im\{-\overline{r}\ell_{n,n-1}-[(n+1) \overline{r} +  {iu}]\overline{\alpha}_n{\alpha}_{n-1} +[(n-1) \overline{r} + ui]  \alpha_{n-2}\overline{\alpha}_{n-1}+\overline{r}\ell_{n,n-1}\} = 0. 
\end{align*}
From item (a) we obtain
$\Re(ur)+\Im\{[(n-1) \overline{r} + ui]  \alpha_{n-2}\overline{\alpha}_{n-1}+\overline{r}\ell_{n,n-1}\} = 0.$
From \eqref{prop-coef-n-1}, by replacing $\overline{\alpha}_{n-1}\alpha_{n-2}=\ell_{n,n-1}-\ell_{n-1,n-1}$, we have 
$\Im[(n\overline{r}+iu)\ell_{n,n-1}]-\Im\{[(n-1)\overline{r}+iu]\ell_{n-1,n-2}\}=-\Re(ur)$. We use this identity recursively with $\ell_{0,-1}=0$ to obtain (b). Finally, (c) follows from (b) for $n$ and $n+1$. 
\end{proof}

\begin{corollary}
\label{Thm-aplic-medida-exemplo3-JAT} Let $\Phi_n$ be the MOPUC with respect the weight function \eqref{Exemplo3-JAT-2013}. Then
 \begin{align*}
      z(z&-r)(\overline{r}z-1)  \Phi_{n}''(z) \\
     & +\{[\overline{r}(3-n)+iu]z^2+[(n-2)(1+|r|^2)-2\Re(ur)i]z-r(n-1)-\overline{iu}\}\Phi_n'(z) \\
     & -\{ n(iu+2\overline{r})z - [\overline{r}(n+1)+iu]\ell_{n+1,n}+n[2\Re(ur)i-1-|r|^2]
     \}\Phi_n(z) \\
     & = -[\overline{r}(n+1)+iu]\overline{\alpha}_n\Phi_n^{\ast}(z). 
 \end{align*}
\end{corollary}
\begin{proof}Lemma \ref{identi-exem3-jaat2023}, item (a), can be used to write
\begin{align*}
    Y_{n}(z)U_{n}(z)&-V_{n}(z)W_{n}(z) \\  
    &=\frac{z}{|r|^2} \times  
    \Big\{
    iun\overline{r}z^3-\{
iun(1+|r|^2)+\overline{r}\ell_{n,n-1}(\overline{r}n+iu)+\overline{r}n2i\Re(ur)\}z^2 \\
    & + \{n(1+|r|^2)[r\overline{\ell}_{n+1,n}+(\overline{iu}+rn)\alpha_{n}\overline{\alpha}_{n-1}]+iunr+|\overline{r}\ell_{n+1,n}+(\overline{r}n+iu)\overline{\alpha}_n\alpha_{n-1}|^2 \\
    &-|\overline{r}(n+1)+iu|^2|\alpha_n|^2-|\overline{r}n+iu|^2|\alpha_{n-1}|^2+|u|^2\}z-r(nr+\overline{iu})\overline{\ell}_{n,n-1} \Big\}. 
\end{align*}
Notice that 
    \begin{align*}
        \widehat{A}(z)=\frac{z}{|r|^2}\times (\overline{r})^2A(z). 
    \end{align*}
By defining
  \begin{align*}
      \mathfrak{r}_{n}  =& n(1+|r|^2)[r\overline{\ell}_{n+1,n}+(rn+\overline{iu})\alpha_n\overline{\alpha}_{n-1}] + |\overline{r}\ell_{n+1,n}+(\overline{r}n+iu)\overline{\alpha}_n\alpha_{n-1}|^2 \\
          & -|\overline{r}(n+1)+iu|^2|\alpha_n|^2-|\overline{r}n+iu|^2|\alpha_{n-1}|^2 +|u|^2\\
          &-(|r|^2+1)[\ell_{n,n-1}(\overline{r}n+iu)+2ni\Re(ur)], 
  \end{align*}
one can write 
  \begin{align*}
  Y_{n}(z)U_{n}(z)- V_{n}(z)&W_{n}(z)  \\ 
    & = \frac{z}{|r|^2}\bigg\{\overline{r}A(z)[iunz - \ell_{n,n-1}(\overline{r}n+iu)+i2n\Re(ur)]+ \\
  &\mathfrak{r}_{n}z + r\{\ell_{n,n-1}(\overline{r}n+iu)-\overline{\ell}_{n,n-1}(rn+\overline{iu})+2ni\Re(ur)\}\bigg\}. 
  \end{align*}
From Lemma \ref{identi-exem3-jaat2023}, item (b), we get
  \begin{align*}
      \frac{Y_{n}(z)U_{n}(z)-V_{n}(z)W_{n}(z)}{\widehat{A}(z)} = \frac{1}{(\overline{r})^2}\times
      \Big\{\overline{r}[iunz - \ell_{n,n-1}(\overline{r}n+iu)+i2n\Re(ur)]
      +\frac{\mathfrak{r}_n z}{A(z)}\Big\}. 
  \end{align*}
Therefore, from Theorem \ref{Thm3-second-order-1}, we obtain
\begin{align*}
      z(z&-r)(\overline{r}z-1)  \Phi_{n}''(z) \\
     & +\{[\overline{r}(3-n)+iu]z^2+[(n-2)(1+|r|^2)-2\Re(ur)i]z-r(n-1)-\overline{iu}\}\Phi_n'(z) \\
     & -\{ n(iu+2\overline{r})z - [\overline{r}(n+1)+iu]\ell_{n+1,n}+n[2i\Re(ur)-1-|r|^2]
     \}\Phi_n(z) \\
     &+[\overline{r}(n+1)+iu]\overline{\alpha}_n\Phi_n^{\ast}(z) = \frac{\mathfrak{r}_nz}{\overline{r}A(z)}\Phi_n(z). 
 \end{align*}
Since the polynomial $A(z)$ has one zero outside the unit disc and all zeros of $\Phi_n$ are inside the unit disk, we conclude that $\mathfrak{r}_n=0$. 
\end{proof}

\section{Final remarks and numerical simulations}
\label{final-section}

In this work we provided explicit first and second order differential equations for semiclassical MOPUC belonging to the class $(p,q)$, with $p\leqslant 3$ and $q\leqslant 3$. As a consequence, we obtained difference equations for the Verblunsky coefficients.
It is noteworthy that the Verblunsky coefficients are explicitly known only for a few MOPUC, as some examples discussed in Section  \ref{Example-ranga-WF}  and Subsection \ref{Subsec-JP}. Therefore, discrete equations are interesting tools to study properties of such coefficients, and in what follows we show how our results can be used to generate the coefficients numerically.

Let us consider the case discussed in Section \ref{Subsec-42-example2}.
for $w(\theta)=e^{t\cos(\theta)}e^{-\eta\theta}[\sin^2(\theta/2)]^\lambda$. For this measure, if $\eta=0$, the Verblunsky coefficients are real and satisfy the discrete Painlevé II equation \eqref{tipo-painleve-absolute-value}, which therefore can be used to generate the $\alpha_n$ when suitable initial conditions are provided.  In fact, let  $\mu_{n}=\int_{0}^{2\pi}e^{-in\theta}w(\theta)d\theta$ be the measure $n$-th moment.  Applying Heine's formula \cite{Simon2005} we obtain the initial conditions
\begin{align*}
    \alpha_{0}={\mu}_{1}/{\mu}_{0} , \ \mbox{and} \ \alpha_{1}=[{\mu}_{0}{\mu}_{2}-({\mu}_{1})^2]/({\mu}_{0}^{2}-\mu_{1}\overline{\mu}_{1}),
\end{align*}
and, for $n\geqslant 2$, the $\alpha_{n}$ are given recursively by
\begin{equation}\label{form-numerica}
    \alpha_{n}=-\frac{2}{t}\frac{[\lambda+(\lambda+n)\alpha_{n-1}]}{1-\alpha_{n-1}^2}-\alpha_{n-2}.
\end{equation}
In the Figure \ref{fig1a}, we present the Verblunsky coefficients obtained from \eqref{form-numerica}. We used as parameters $\lambda=0$, $\lambda=5$, $\lambda=10$, with $t=10$ fixed, and generated $\alpha_n$ using the software Wolfram Mathematica. 

A similar idea can be used for different discrete equations. For example, consider using the discrete equation from Theorem \ref{Theo-verb-new-new} to generate $\alpha_n$ for $n\geqslant 3$. If $\eta\neq 0$, the Verblunsky coefficients are complex. Again we use Heine's formula to find the initial conditions $\alpha_0$, $\alpha_{1}$, and
\begin{align*}
\alpha_{2}=\left|
\begin{array}{ccc}
\mu_1 & \mu_2 & \mu_3 \\
\mu_0 & \mu_1 & \mu_2 \\
\overline{\mu}_1 & \mu_0 & \mu_1 \\
\end{array}
\right|\left|
\begin{array}{ccc}
\mu_0 & \overline{\mu}_1 & \overline{\mu}_2 \\
\mu_1 & \mu_0 & \overline{\mu}_1 \\
\mu_2 & \mu_1 & \mu_0 \\
\end{array}
\right|^{-1}.
\end{align*}
The Figure \ref{fig1b} shows the behavior of $|\alpha_n|$ for $\eta=0$, $\eta=5$, and $\eta=10$ with $t=10$, $\lambda=5$. As expected, in both simulations we obtain $|\alpha_n|<1$ for $n=0,1,\ldots,30$. 

The drawback of using Heine's formula instead of discrete equation is that in order to find the value of $\alpha_n$ one needs to compute the moments $\mu_{0},\mu_{1},\ldots,\mu_{n+1}$ and to calculate two determinants of $(n+1)\times (n+1)$ matrices.

\begin{figure}[H]
    \centering
      \begin{minipage}[b]{0.49\linewidth}
        \centering
        \includegraphics[width=\linewidth]{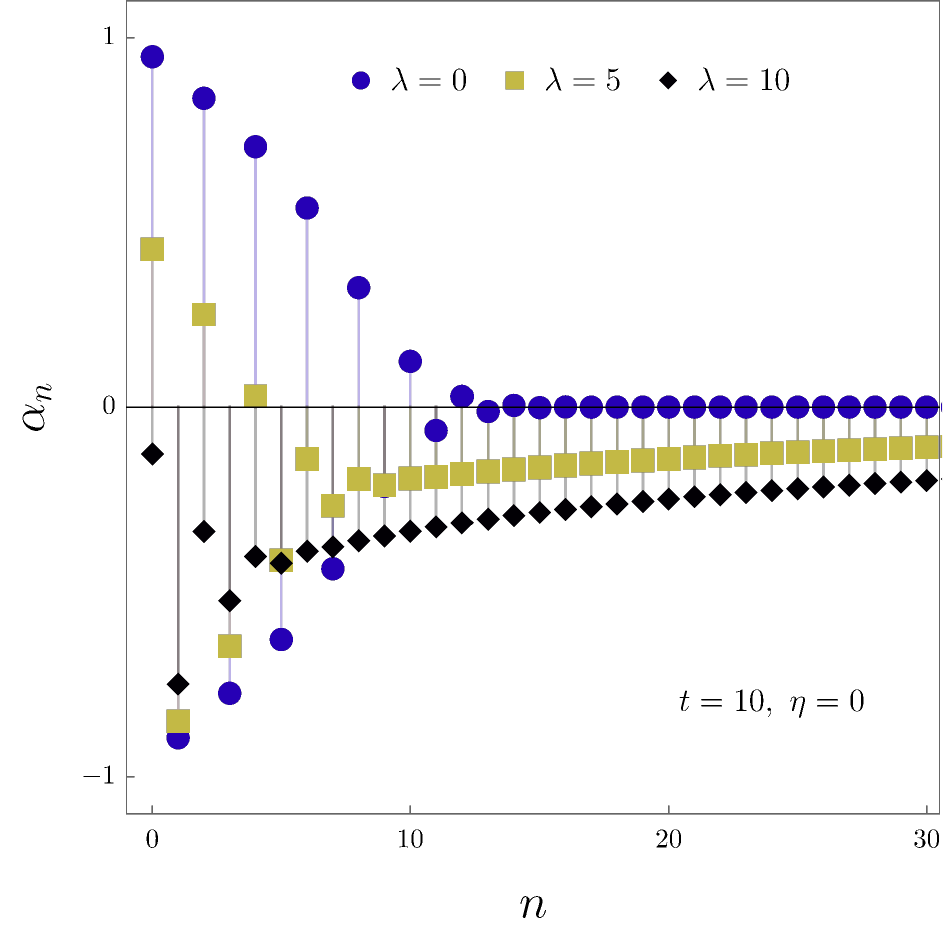} 
         \subcaption{Simulations of $\alpha_n$ using discrete equation \eqref{form-numerica}.} 
        \label{fig1a}
    \end{minipage}
    \hfill
       \begin{minipage}[b]{0.48\linewidth}
        \centering
    \includegraphics[width=\linewidth]{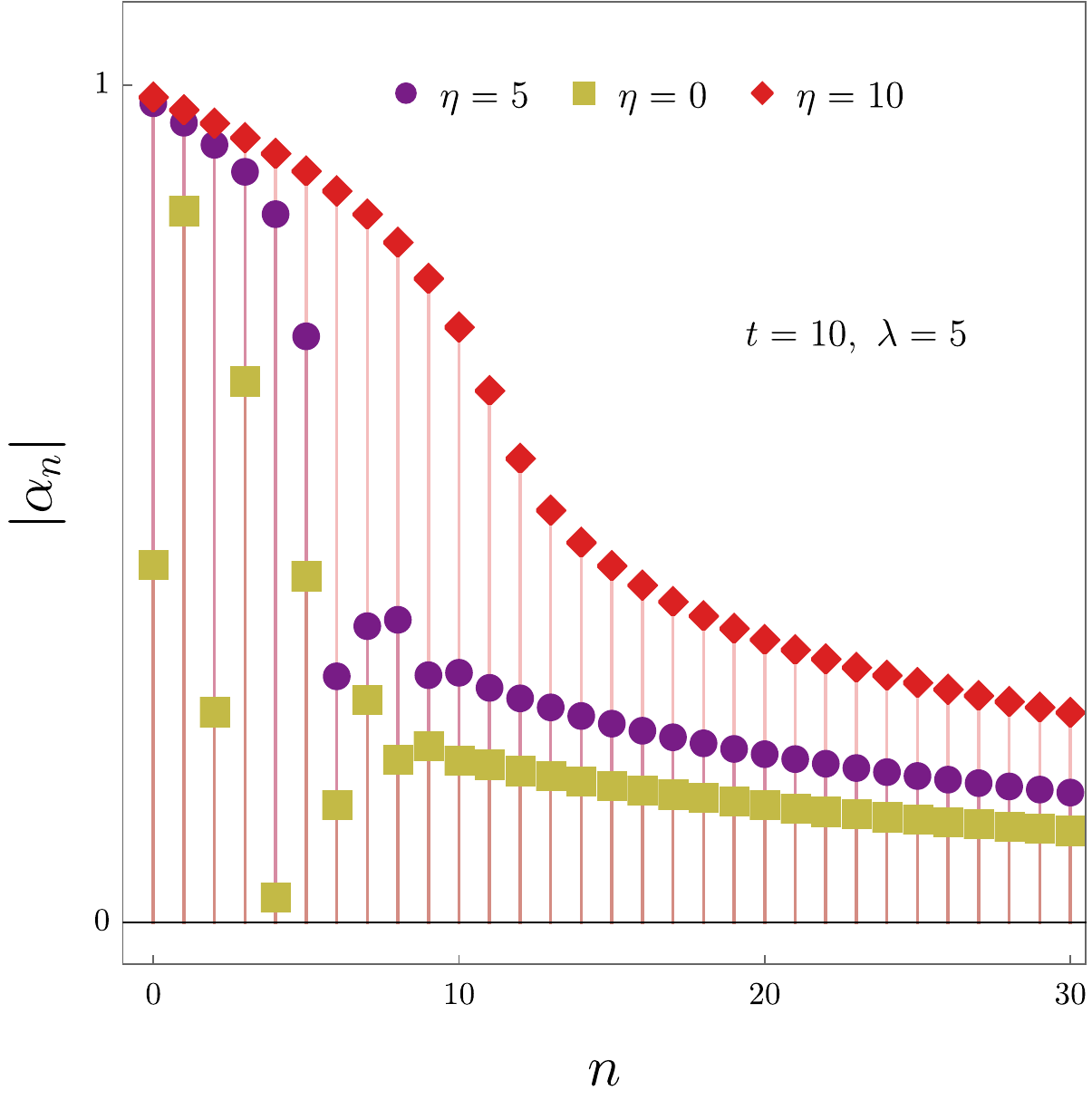}
     \subcaption{Simulations of $|\alpha_n|$ using Theorem \ref{Theo-verb-new-new}.}
    \label{fig1b}
    \end{minipage} 
   \caption{Verblunsky coefficients.}
\end{figure}

\section*{Acknowledgments}
This work was supported by Fundação de Amparo à Pesquisa do Estado de Minas Gerais (FAPEMIG), Brazil (grant number APQ-01574-24).

 \bibliographystyle{plain} 
 \bibliography{refs.bib}

\end{document}